\documentclass[12pt]{article} 
\usepackage{amscd,amsmath,amssymb,amsfonts,braket,mathrsfs,latexsym,enumerate,amsthm}
\usepackage[all]{xy}
\makeatletter
\@namedef{subjclassname@2010}{%
  \textup{2010} Mathematics Subject Classification}
\makeatother

\newtheorem{thm}{Theorem} 
\newtheorem{lem}[thm]{Lemma} 
\newtheorem{cor}[thm]{Corollary}
\newtheorem{prop}[thm]{Proposition}
\newtheorem{conj}[thm]{Conjecture}
\theoremstyle{definition}
\newtheorem{rmk}[thm]{Remark}
\newtheorem{defn}[thm]{Definition}

\numberwithin{thm}{section}
\numberwithin{equation}{section}

\newcommand{\C}{\mathbb{C}}
\newcommand{\F}{\mathbb{F}}

\newcommand{\Fbar}{\overline{F}}
\newcommand{\Q}{\mathbb{Q}}
\newcommand{\Qb}{\overline{\mathbb{Q}}}

\newcommand{\Kb}{\overline{K}}
\newcommand{\kb}{\overline{k}}

\newcommand{\Fsep}{F^{\text{sep}}}
\newcommand{\rhob}{\overline{\rho}}
\newcommand{\chib}{\overline{\chi}}

\newcommand{\Fab}{F^{\text{ab}}}
\newcommand{\Gab}{\mathrm{G}^{\text{ab}}}
\newcommand{\lambdaab}{\lambda^{\text{ab}}}
\newcommand{\trab}{\mathrm{tr}^{\text{ab}}}

\newcommand{\A}{\mathbb{A}}

\newcommand{\kA}{k_{\A}}
\newcommand{\KA}{K_{\A}}

\newcommand{\R}{\mathbb{R}}
\newcommand{\Z}{\mathbb{Z}}
\newcommand{\pP}{\mathbb{P}}

\newcommand{\betab}{\overline{\beta}}

\newcommand{\disc}{\mathrm{disc}\, }

\newcommand{\mfa}{\mathfrak{a}}
\newcommand{\mfp}{\mathfrak{p}}
\newcommand{\mfq}{\mathfrak{q}}
\newcommand{\mfI}{\mathfrak{I}}
\newcommand{\mfP}{\mathfrak{P}}
\newcommand{\mfS}{\mathfrak{S}}

\newcommand{\mfl}{\mathfrak{l}}

\newcommand{\ktil}{\widetilde{k}}
\newcommand{\sA}{\mathscr{A}}
\newcommand{\h}{\mathbb{H}}

\newcommand{\cP}{\mathcal{P}}

\newcommand{\cO}{\mathcal{O}}

\newcommand{\cQM}{\mathcal{QM}}
\newcommand{\cM}{\mathcal{M}}
\newcommand{\cN}{\mathcal{N}}
\newcommand{\cFR}{\mathcal{FR}}
\newcommand{\cK}{\mathcal{K}}
\newcommand{\cS}{\mathcal{S}}
\newcommand{\cT}{\mathcal{T}}

\newcommand{\Ram}{\mathbf{Ram}}

\newcommand{\id}{\mathrm{id}}

\newcommand{\pmu}{\pmb{\mu}}

\newcommand{\Frob}{\mathrm{Frob}}

\newcommand{\Gal}{\mathrm{Gal}}
\newcommand{\G}{\mathrm{G}}
\newcommand{\M}{\mathrm{M}}
\newcommand{\N}{\mathrm{N}}
\newcommand{\GL}{\mathrm{GL}}

\newcommand{\GSp}{\mathrm{GSp}}
\newcommand{\SL}{\mathrm{SL}}
\newcommand{\Hom}{\mathrm{Hom}}
\newcommand{\End}{\mathrm{End}}
\newcommand{\Aut}{\mathrm{Aut}}

\newcommand{\im}{\mathrm{Im}\, }
\newcommand{\Norm}{\mathrm{Norm}}

\newcommand{\Nrd}{\mathrm{Nrd}}

\newcommand{\ord}{\mathrm{ord}\, }

\newcommand{\tr}{\mathrm{tr}}

\newcommand{\Spec}{\mathrm{Spec}}

\newcommand{\Br}{\mathrm{Br}}
\newcommand{\cf}{cf.\ }
\newcommand{\inj}{\hookrightarrow}

\newcommand{\resp}{resp.\ }

\newcommand{\ch}{\mathrm{char}\,}

\begin{document}

\title{Algebraic points on Shimura curves of $\Gamma_0(p)$-type}
\author{Keisuke Arai and Fumiyuki Momose}
\date{}



%
%

%
%



\maketitle

\begin{center}
\textit{To the memory of Fumiyuki Momose}
\end{center}

\begin{abstract}

In this article,
we classify the characters associated to algebraic points
on Shimura curves of $\Gamma_0(p)$-type, and over a quadratic field
we show that
there are at most elliptic points on such a Shimura curve
for every sufficiently large prime number $p$.
This is an analogue of the study of rational points or points over a quadratic field
on the modular curve $X_0(p)$ by Mazur and one of the author (Momose).
We also apply the result to a finiteness conjecture on abelian varieties
with constrained prime power torsion by Rasmussen-Tamagawa.

\end{abstract}


\tableofcontents

\section{Introduction}
\label{intro}


A Shimura curve can be considered to be an analogue of a modular curve.
In this context we study points on Shimura curves, and show that
there are few points over quadratic fields on Shimura curves of
``$\Gamma_0(p)$-type" just as in the case of the modular curve
$X_0(p)$ defined below.
We also get an irreducibility result of the mod $p$ Galois representations
associated to abelian surfaces with quaternionic multiplication,
which we apply to a finiteness conjecture on abelian varieties.

For an integer $N\geq 1$,
let $X_0(N)$ be the smooth compactification of the coarse moduli scheme over $\Q$
parameterizing isomorphism classes of
pairs $(E,C)$ where $E$ is an elliptic curve and $C$ is a cyclic subgroup of $E$ of order $N$.
We know that $X_0(N)$ is a proper smooth curve over $\Q$
(\cf \cite[Th\'{e}or\`{e}me 3.4, p.212]{DR})
and that $X_0(1)$ is isomorphic to the projective line
$\pP_{\Q}^1$ (\cf \cite[Th\'{e}or\`{e}me 1.1, p.267]{DR}).
When $N$ is a prime number $p$,
Mazur and one of the authors studied points on the modular curve $X_0(p)$.

\begin{thm}
\label{X0p}

(1) {\rm \cite[Theorem 7.1, p.153]{Ma}}
For a prime number $p$, we have $X_0(p)(\Q)=\{\text{cusps}\}$
if and only if $p\not\in\{2,3,5,7,11,13,17,19,37,43,67,163\}$.

(2) {\rm \cite[Theorem B, p.330]{Mo}}
Let $k$ be a quadratic field which is not an imaginary quadratic field of
class number one.
Then there is a finite set $\mfS(k)$ of prime numbers depending on $k$
such that $X_0(p)(k)=\{\text{cusps}\}$ holds for every prime number $p\not\in\mfS(k)$.

\end{thm}

Notice that all the imaginary quadratic fields of class number one
are 
$\Q(\sqrt{-1})$, $\Q(\sqrt{-2})$, $\Q(\sqrt{-3})$,
$\Q(\sqrt{-7})$, $\Q(\sqrt{-11})$, $\Q(\sqrt{-19})$,
$\Q(\sqrt{-43})$, $\Q(\sqrt{-67})$ and $\Q(\sqrt{-163})$
by \cite[p.205]{Ba} or \cite[Theorem, p.2]{St}.
In Theorem \ref{X0p} (2), the finite set $\mfS(k)$ is effectively
estimated except at most one prime number.
If such a prime exists, it is concerned with a Siegel zero
of the $L$-functions of quadratic characters (\cf Theorem \ref{goldfeld}).
For related topics about modular curves, see \cite{A2}.

We have an analogue of Theorem \ref{X0p} for Shimura curves, as
explained below.
Let $B$ be an indefinite quaternion division algebra over $\Q$.
Let
$$d:=\disc B$$
be the discriminant of $B$.
Then $d$ is the product of an even number of distinct prime numbers, and $d>1$.
Fix a maximal order $\cO$ of $B$.
For each prime number $p$ not dividing $d$, fix an isomorphism
\begin{equation}
\label{OM2}
\cO\otimes_{\Z}\Z_p\cong\M_2(\Z_p) \nonumber
\leqno(1.1)
\end{equation}
of $\Z_p$-algebras.

\begin{defn}
\label{defqm}
\rm

(\cf \cite[p.591]{Bu})
Let $S$ be a scheme.
A QM-abelian surface by $\cO$
over $S$ is a pair $(A,i)$ where $A$ is an 
abelian surface over $S$ (i.e. $A$ is an abelian scheme over $S$ 
of relative dimension $2$), and 
$i:\cO\inj\End_S(A)$ 
is an injective ring homomorphism (sending $1$ to $\id$). 
We consider that $A$ has a left $\cO$-action.
We sometimes omit ``by $\cO$" and simply write ``a QM-abelian surface".

\end{defn}

Let $M^B$ be the coarse moduli scheme over $\Q$ parameterizing isomorphism classes
of QM-abelian surfaces by $\cO$.
The notation $M^B$ is permissible
although we should write $M^{\cO}$ instead of $M^B$;
for even if we replace $\cO$ by another maximal order $\cO'$,
we have a natural isomorphism
$M^{\cO}\cong M^{\cO'}$
since $\cO$ and $\cO'$ are conjugate in $B$.
Then $M^B$ is a proper smooth curve over $\Q$, called a Shimura curve.
For a prime number $p$ not dividing $d$,
let $M_0^B(p)$
be the coarse moduli scheme over $\Q$ parameterizing isomorphism classes
of triples $(A,i,V)$ where $(A,i)$ is a QM-abelian surface by $\cO$
and $V$ is a left $\cO$-submodule of $A[p]$ with $\F_p$-dimension $2$.
Here $A[p]$ is the kernel of multiplication by $p$ in $A$.
Then $M_0^B(p)$ is a proper smooth curve over $\Q$, which we call a
Shimura curve of $\Gamma_0(p)$-type.
We have a natural map $$\pi^B(p):M_0^B(p)\longrightarrow M^B$$
over $\Q$ defined by $(A,i,V)\longmapsto (A,i)$.
Note that $M^B$ (\resp $M_0^B(p)$) is an analogue of the modular curve $X_0(1)$
(\resp $X_0(p)$).

We study points on $M_0^B(p)$, and over a quadratic field
(considered as a subfield of $\C$)
we show that there are at most 
elliptic points of order $2$ or $3$
on it for every sufficiently large prime number $p$.
The following is the main result of this article.

\begin{thm}
\label{mainthm0}

Let $k$ be a quadratic field which is not an imaginary quadratic field of
class number one.
Then there is a finite set $\cN(k)$ of prime numbers depending on $k$
satisfying the following.

(1)
If $B\otimes_{\Q}k\cong\M_2(k)$, then $M_0^B(p)(k)=\emptyset$ holds
for every prime number $p\not\in\cN(k)$ not dividing $d$.

(2)
If $B\otimes_{\Q}k\not\cong\M_2(k)$, then
$M_0^B(p)(k)\subseteq\{\text{elliptic points of order $2$ or $3$}\}$
holds for every prime number $p\not\in\cN(k)$ not dividing $d$.

\end{thm}

\begin{rmk}
\label{N(k)effective}
\rm

The set $\cN(k)$ in Theorem \ref{mainthm0} is effectively estimated except at most one prime number
(\cf Theorem \ref{goldfeld}).

\end{rmk}

We know $M^B(\R)=\emptyset$ by \cite[Theorem 0, p.136]{Sh}.
Since there is a map
$\pi^B(p):M_0^B(p)\longrightarrow M^B$
over $\Q$, we have $M_0^B(p)(\R)=\emptyset$.
So if $k$ has a real place, then we have already known $M_0^B(p)(k)=\emptyset$
before Theorem \ref{mainthm0}.

In Section \ref{QM}, we discuss
the Galois representations associated to QM-abelian surfaces.
In Section \ref{endaut}, we discuss the endomorphism rings and the automorphism groups
of QM-abelian surfaces.
In Section \ref{fieldofdefinition}, we study the fields of definition
of the pairs (\resp the triples) corresponding to
algebraic points on $M^B$ (\resp $M_0^B(p)$).
In Section \ref{charI} and Section \ref{charII},
we classify the characters associated to algebraic points on $M_0^B(p)$
by slightly modifying the method in \cite{Mo},
and show Theorem \ref{mainthm0}.
In Section \ref{ex}, we give examples of points over imaginary quadratic fields
on the Shimura curves $M^B$ of genus zero. 
In Section \ref{galim}, we study the images of the mod $p$ Galois representations associated to
QM-abelian surfaces over imaginary quadratic fields.
In Section \ref{apply}, we apply the result on the Galois images in Section \ref{galim}
to a finiteness conjecture
on abelian varieties in \cite{RT}.

The first author (Arai) is very sorry for the loss of the coauthor (Momose)
during this work.
This article shall be dedicated to Fumiyuki Momose.

The first author would like to thank Yuichiro Taguchi
for suggesting an application to a finiteness conjecture on abelian varieties.
He would like to thank also Yoshiyasu Ozeki and
Akio Tamagawa 
for helpful comments.

\vspace{5mm}
\noindent
{\bf Notation}

\vspace{1mm}

For an integer $n\geq 1$ and a commutative group (or a commutative group scheme) $G$,
let $G[n]$ denote the kernel of multiplication by $n$ in $G$.
For a field $F$,
let $\ch F$ denote the characteristic of $F$,
let $\Fbar$ denote an algebraic closure of $F$,
let $\Fsep$ (\resp $\Fab$) denote the separable closure
(\resp the maximal abelian extension) of $F$ inside $\Fbar$,
and let $\G_F=\Gal(\Fsep/F)$, $\Gab_F=\Gal(\Fab/F)$.
For a number field $k$, 
let $h_k$ denote the class number of $k$;
fix an inclusion $k\hookrightarrow\C$
and take the algebraic closure $\kb$ inside $\C$;
let $k_v$ denote the completion of $k$ at $v$
where $v$ is a place (or a prime) of $k$;
let $k_{\A}$ denote the ad\`{e}le ring of $k$;
and let $\Ram (k)$ denote the set of prime numbers which are ramified in $k$.
For a number field or a local field $k$, let $\cO_k$ denote
the ring of integers of $k$.
For a scheme $S$ and an abelian scheme $A$ over $S$,
let $\End_S(A)$ denote the ring of endomorphisms of $A$ defined over $S$.
If $S=\Spec(F)$ for a field $F$ and if $F'/F$ is a field extension,
simply put
$\End_{F'}(A):=\End_{\Spec(F')}(A\times_{\Spec(F)}\Spec(F'))$
and
$\End(A):=\End_{\Fbar}(A)$.
Let $\Aut(A):=\Aut_{\Fbar}(A)$ be the group
of automorphisms of $A$ defined over $\Fbar$.
For a prime number $p$ and an abelian variety $A$ over a field $F$,
let
$\displaystyle T_pA:=\lim_{\longleftarrow}A[p^n](\Fbar)$
be the $p$-adic Tate module of $A$,
where the inverse limit is taken with respect to
multiplication by $p$ : $A[p^{n+1}](\Fbar)\longrightarrow A[p^n](\Fbar)$.

\section{Galois representations associated to QM-abelian surfaces (generalities)}
\label{QM}

We consider the Galois representation associated to a QM-abelian surface.
Take a prime number $p$ not dividing $d$.
Let $F$ be a field with $\ch F\ne p$.
Let $(A,i)$ be a QM-abelian surface by $\cO$ over $F$.
We have isomorphisms of $\Z_p$-modules:
$$\Z_p^4\cong T_pA\cong\cO\otimes_{\Z}\Z_p\cong\M_2(\Z_p).$$
The middle is also an isomorphism of left $\cO$-modules;
the last is also an isomorphism of $\Z_p$-algebras (which is fixed in (\ref{OM2})).
We sometimes identify these $\Z_p$-modules.
Take a $\Z_p$-basis
$$e_1=\left(\begin{matrix} 1&0 \\ 0&0\end{matrix}\right),\ 
e_2=\left(\begin{matrix} 0&0 \\ 1&0\end{matrix}\right),\ 
e_3=\left(\begin{matrix} 0&1 \\ 0&0\end{matrix}\right),\ 
e_4=\left(\begin{matrix} 0&0 \\ 0&1\end{matrix}\right)$$
of $\M_2(\Z_p)$.
Then the image of the natural map
$$\M_2(\Z_p)\cong\cO\otimes_{\Z}\Z_p\hookrightarrow\End(T_pA)\cong\M_4(\Z_p)$$
lies in
$\left\{\left(\begin{matrix} 
X\ &0 \\ 0\ &X
\end{matrix}\right)
\Bigg|\, X\in\M_2(\Z_p)\right\}$.
%
The action of the Galois group $\G_F$ on $T_pA$ induces a representation
$$\rho:\G_F\longrightarrow\Aut_{\cO\otimes_{\Z}\Z_p}(T_pA)\subseteq\Aut(T_pA)
\cong\GL_4(\Z_p),$$
where
$\Aut_{\cO\otimes_{\Z}\Z_p}(T_pA)$
is the group of automorphisms of $T_pA$
commuting with the action of $\cO\otimes_{\Z}\Z_p$.
We often identify
$\Aut(T_pA)=\GL_4(\Z_p)$.
The above observation implies
$$\Aut_{\cO\otimes_{\Z}\Z_p}(T_pA)=
\left\{\left(\begin{matrix} sI_2&tI_2 \\ uI_2&vI_2\end{matrix}\right) \Bigg|\,
\left(\begin{matrix} s&t \\ u&v\end{matrix}\right)\in\GL_2(\Z_p)\right\},$$
where 
$I_2=\left(\begin{matrix} 1\ &0 \\ 0\ &1\end{matrix}\right)$.
Note that $\Aut_{\cO\otimes_{\Z}\Z_p}(T_pA)$ is contained in the general symplectic group
$\GSp_4(\Z_p)=\{Y\in\GL_4(\Z_p)\mid
{}^tYJYJ^{-1}\in\Z_p^{\times}I_4\}$,
where
$J=\left(\begin{matrix} 0&-I_2 \\ I_2&0\end{matrix}\right)$
and
$I_4=\left(\begin{matrix} I_2&0 \\ 0&I_2\end{matrix}\right)$.
Then the Galois representation $\rho$ factors as
$$\rho:\G_F\longrightarrow
\left\{\left(\begin{matrix} sI_2&tI_2 \\ uI_2&vI_2\end{matrix}\right) \Bigg|\,
\left(\begin{matrix} s&t \\ u&v\end{matrix}\right)\in\GL_2(\Z_p)\right\}
\subseteq\GL_4(\Z_p).$$
Let 
$$\rhob:\G_F\longrightarrow
\left\{\left(\begin{matrix} sI_2&tI_2 \\ uI_2&vI_2\end{matrix}\right) \Bigg|\,
\left(\begin{matrix} s&t \\ u&v\end{matrix}\right)\in\GL_2(\F_p)\right\}
\subseteq\GL_4(\F_p)$$
be the reduction of $\rho$ modulo $p$.
Let
\begin{equation}
\label{rhobar}
\rhob_{A,p}:\G_F\longrightarrow\GL_2(\F_p) \nonumber
\leqno(2.1)
\end{equation}
denote the Galois representation induced from $\rhob$ by
$``\left(\begin{matrix} s&t \\ u&v\end{matrix}\right)"$,
so that we have
$\rhob_{A,p}(\sigma)=\left(\begin{matrix} s&t \\ u&v\end{matrix}\right)$
if $\rhob(\sigma)=\left(\begin{matrix} sI_2&tI_2 \\ uI_2&vI_2\end{matrix}\right)$
for $\sigma\in\G_F$.

%

Suppose that $A[p](\Fsep)$ has a left $\cO$-submodule $V$ with $\F_p$-dimension $2$
which is stable under the action of $\G_F$.
We may assume
$V=\F_p e_1\oplus\F_p e_2
=\left\{\left(\begin{matrix} *&0 \\ *&0 \end{matrix}\right)\right\}$.
Since $V$ is stable under the action of $\G_F$, we find
$\rhob_{A,p}(\G_F)\subseteq
\left\{\left(\begin{matrix} 
s&t \\ 0&v
\end{matrix}\right)\right\}\subseteq\GL_2(\F_p)$.
%
Let
\begin{equation}
\label{lambda}
\lambda:\G_F\longrightarrow\F_p^{\times} \nonumber
\leqno(2.2)
\end{equation}
denote the character induced from $\rhob_{A,p}$ by ``$s$", so that
$\rhob_{A,p}(\sigma)=
\left(\begin{matrix} 
\lambda(\sigma)&* \\ 0&*
\end{matrix}\right)$
for $\sigma\in\G_F$.
Note that
$\G_F$ acts on $V$ by $\lambda$
(i.e. $\rhob(\sigma)(v)=\lambda(\sigma)v$
for $\sigma\in\G_F$, $v\in V$).

\begin{rmk}
\label{analogellip}

The Galois representation $\rhob_{A,p}$ as above
looks like that of an elliptic curve over $F$ having an $F$-rational
isogeny of degree $p$.

\end{rmk}

\section{Endomorphism rings and automorphism groups}
\label{endaut}

We recall the notion of CM (complex multiplication) on an abelian variety.
Let $F$ be a field, and let $A$ be an abelian variety over $F$.
The abelian variety $A$ is said to have CM (over $\Fbar$)
if $\End(A)\otimes_{\Z} \Q$ contains a product $R$ of number fields satisfying
$\dim_{\Q} R=2\dim A$.

Consider the case where $\ch F=0$.
If $A/F$ is $\Fbar$-simple and has CM by $R$ as above, then $\End(A)\otimes_{\Z}\Q\cong R$
(\cite[Chapter IV, Section 21, Table, p.202]{Mu}).
If $(A,i)/F$ is a QM-abelian surface, then either $A$ has CM or $A$ is $\Fbar$-simple.
If $(A,i)/F$ is a QM-abelian surface with CM, then $A$ is $\Fbar$-isogenous to $E\times E$
where $E$ is an elliptic curve over $\Fbar$ with CM.

Next we consider the automorphism group of a QM-abelian surface.
Let $(A,i)$ be a QM-abelian surface by $\cO$ over a field $F$.
Put
$$\End_{\cO}(A):=\{f\in\End(A)\mid fi(g)=i(g)f \text{\ \ for any $g\in\cO$}\}$$
and
$$\Aut_{\cO}(A):=\Aut(A)\cap\End_{\cO}(A).$$
Assume $\ch F=0$.
If $A$ has no CM, then $\End_F(A)=\End(A)=\cO$ (loc. cit.),
$\End_{\cO}(A)=\Z$
and $\Aut_{\cO}(A)=\{\pm 1\}\cong\Z/2\Z$.
If $A$ has CM, then
$\End(A)\otimes_{\Z}\Q=\M_2(\cK)$
for an imaginary quadratic field $\cK$.
In this case $\End_{\cO}(A)$ is an order of the imaginary quadratic field
$\Set{\left(\begin{matrix} a&0 \\ 0&a\end{matrix}\right)\in\M_2(\cK)|a\in\cK}$.
Then $\Aut_{\cO}(A)\cong\Z/2\Z$, $\Z/4\Z$ or $\Z/6\Z$.

Let $p$ be a prime number not dividing $d$.
Let $(A,i,V)$ be a triple where $(A,i)$ is a QM-abelian surface by $\cO$ over
a field $F$
and $V$ is a left $\cO$-submodule of $A[p](\Fbar)$ with $\F_p$-dimension $2$.
Define a subgroup $\Aut_{\cO}(A,V)$ of $\Aut_{\cO}(A)$ by
$$\Aut_{\cO}(A,V):=\{f\in\Aut_{\cO}(A)\mid f(V)=V\}.$$
Assume $\ch F=0$.
Then $\Aut_{\cO}(A,V)\cong\Z/2\Z$, $\Z/4\Z$ or $\Z/6\Z$.
Notice that we have
$\Aut_{\cO}(A)\cong\Z/2\Z$
(\resp $\Aut_{\cO}(A,V)\cong\Z/2\Z$)
if and only if
$\Aut_{\cO}(A)=\{\pm 1\}$
(\resp $\Aut_{\cO}(A,V)=\{\pm 1\}$).
Notice also that if $A$ has no CM, then $\Aut_{\cO}(A)=\Aut_{\cO}(A,V)=\{\pm 1\}$.

We express the sets $M^B(\C)$ and $M_0^B(p)(\C)$ of $\C$-valued points
as quotients of the upper half-plane.
Let $\cO'$ be the order of $B$
satisfying $\cO'\subseteq\cO$ and
\begin{equation*}
\begin{cases}
\cO'\otimes_{\Z}\Z_l=\cO\otimes_{\Z}\Z_l \ \ (l\ne p), \\
\cO'\otimes_{\Z}\Z_p
=\Set{\left(\begin{matrix} s&t \\ u&v \end{matrix}\right)\in\M_2(\Z_p)
\cong\cO\otimes_{\Z}\Z_p|u\in p\Z_p}.
\end{cases}
\end{equation*}
Consider the groups
$$\Gamma:=\{x\in\cO\mid\Nrd(x)=1\}$$
and
$$\Gamma':=\{x\in\cO'\mid\Nrd(x)=1\},$$
where $\Nrd$ is the reduced norm.
Then $\Gamma$ and $\Gamma'$ are considered to be subgroups of $\SL_2(\R)$
via $\SL_2(\R)\subseteq \GL_2(\R)\cong (B\otimes_{\Q}\R)^{\times}$.
We have isomorphisms of Riemann surfaces
$$M^B(\C)\cong\Gamma\backslash\h$$
and
$$M_0^B(p)(\C)\cong\Gamma'\backslash\h,$$
where $\h$ is the upper half-plane
(\cf \cite[4.3, p.669]{Be}).
A point of $M_0^B(p)(\C)$ is an elliptic point of order $2$ (\resp $3$)
if and only if the corresponding triple $(A,i,V)$ over $\C$ satisfies
$\Aut_{\cO}(A,V)\cong\Z/4\Z$ (\resp $\Aut_{\cO}(A,V)\cong\Z/6\Z$)
(\cf \cite[1.2, p.9]{Sh0}).
Let $\nu_2(p)$ (\resp $\nu_3(p)$) be the number of
elliptic points of order $2$ (\resp $3$) of $M_0^B(p)(\C)$.
Then we have
$$\nu_2(p)
=\left(1+\left(\frac{-1}{p}\right)\right)\prod_{l|d}\left(1-\left(\frac{-1}{l}\right)\right),$$
$$\nu_3(p)
=\left(1+\left(\frac{-3}{p}\right)\right)\prod_{l|d}\left(1-\left(\frac{-3}{l}\right)\right)
$$
where $l$ is a prime divisor of $d$ and
\begin{align*}
\left(\frac{-1}{q}\right)&=
\begin{cases}
1 &\text{if $q\equiv 1\mod{4}$},\\
-1 &\text{if $q\equiv -1\mod{4}$},\\
0 &\text{if $q=2$},
\end{cases}\\
\left(\frac{-3}{q}\right)&=
\begin{cases}
1 &\text{if $q\equiv 1\mod{3}$},\\
-1 &\text{if $q\equiv -1\mod{3}$},\\
0 &\text{if $q=3$}
\end{cases}
\end{align*}
for a prime number $q$ (\cite[Theorem 3.12, p.111]{Shimizu}).
%
So the isomorphism
$\Aut_{\cO}(A,V)\cong\Z/4\Z$ (\resp $\Aut_{\cO}(A,V)\cong\Z/6\Z$)
is possible only when
$p\not\equiv -1\bmod{4}$ and
all the prime divisors of $d$ are congruent to $2$ or $-1$ modulo $4$
(\resp $p\not\equiv -1\bmod{3}$ and
all the prime divisors of $d$ are congruent to $0$ or $-1$ modulo $3$).

\section{Fields of definition}
\label{fieldofdefinition}

Let $k$ be a number field (considered as a subfield of $\C$).
Let $p$ be a prime number not dividing $d$.
Take a point
$$x\in M_0^B(p)(k).$$
Let $x'\in M^B(k)$ be the image
of $x$ by the map $\pi^B(p):M_0^B(p)\longrightarrow M^B$.
Then $x'$ is represented by a QM-abelian surface (say $(A_x,i_x)$) over $\kb$,
and $x$ is represented by a triple $(A_x,i_x,V_x)$
where $V_x$ is a left $\cO$-submodule of $A_x[p](\kb)$
with $\F_p$-dimension $2$.
For a finite extension $M$ of $k$ (in $\kb$),
we say that we can take $(A_x,i_x)$ (\resp $(A_x,i_x,V_x)$) to be defined
over $M$
if there is a QM-abelian surface $(A,i)$ over $M$ such that
$(A,i)\otimes_M\kb$ is isomorphic to $(A_x,i_x)$
(\resp if there is a QM-abelian surface $(A,i)$ over $M$
and a left $\cO$-submodule $V$ of $A[p](\kb)$ with $\F_p$-dimension $2$
stable under the action of $\G_M$ such that there is an isomorphism
between $(A,i)\otimes_M\kb$ and $(A_x,i_x)$ under which $V$ corresponds
to $V_x$).
Put
$$\Aut(x):=\Aut_{\cO}(A_x,V_x),\ \ \ \ \Aut(x'):=\Aut_{\cO}(A_x).$$
Then $\Aut(x)$ is a subgroup of $\Aut(x')$.
The point $x$ is called a CM point if $A_x$ has CM.
Note that if $\Aut(x')\cong\Z/4\Z$ or $\Z/6\Z$,
then $x$ is a CM point.
In particular, if $x$ is an elliptic point of order $2$ or $3$,
then $x$ is a CM point.

Since $x$ is a $k$-rational point, we have 
$^{\sigma}x=x$ for any $\sigma\in\G_k$.
Then, for any $\sigma\in\G_k$, there is an isomorphism
$$\phi_{\sigma}:{}^{\sigma}(A_x,i_x,V_x)\longrightarrow (A_x,i_x,V_x),$$
which we fix once for all.
Let
$$\phi'_{\sigma}:{}^{\sigma}(A_x,i_x)\longrightarrow (A_x,i_x)$$
be the isomorphism induced from $\phi_{\sigma}$ by forgetting $V_x$.
For $\sigma,\tau\in\G_k$, put
$$c_x(\sigma,\tau)
:=\phi_{\sigma}\circ {}^{\sigma}\phi_{\tau}
\circ\phi_{\sigma\tau}^{-1}\in\Aut(x)$$
and
$$c'_x(\sigma,\tau)
:=\phi'_{\sigma}\circ {}^{\sigma}\phi'_{\tau}
\circ(\phi'_{\sigma\tau})^{-1}\in\Aut(x').$$
Then $c_x$ (\resp $c'_x$) is a $2$-cocycle
and defines a cohomology class
$[c_x]\in H^2(\G_k,\Aut(x))$
(\resp $[c'_x]\in H^2(\G_k,\Aut(x'))$.
Here the action of $\G_k$ on $\Aut(x)$ (\resp $\Aut(x')$)
is defined as follows.

Let $f:A_x\longrightarrow A_x$
be an isomorphism in $\Aut(x)$ (\resp $\Aut(x')$).
For an element $\sigma\in\G_k$,
let $^{\sigma}f:{}^{\sigma}A_x\longrightarrow {}^{\sigma}A_x$
be the isomorphism induced from $f$.
We define $\sigma(f)$ by the commutative diagram
\begin{equation*}
\begin{CD}
^{\sigma}A_x@>\text{$^{\sigma}f$}>> ^{\sigma}A_x \\
@V\text{$\phi_{\sigma}$}VV   @VV\text{$\phi_{\sigma}$}V \\
A_x@>\text{$\sigma(f)$}>> A_x.
\end{CD}
\end{equation*}
Then we get an action of $\G_k$ on the abelian group $\Aut(x)$ (\resp $\Aut(x')$)
by
$$\G_k\longrightarrow\Aut(\Aut(x)):\sigma\longmapsto(f\longmapsto\sigma(f))$$
$$(\text{\resp}\G_k\longrightarrow\Aut(\Aut(x')):\sigma\longmapsto(f\longmapsto\sigma(f))).$$

Note that the action of $\G_k$ on $\Aut(x)$ (\resp $\Aut(x')$)
is trivial if $\Aut(x)=\{\pm 1\}$ (\resp $\Aut(x')=\{\pm 1\}$).
For a cohomology class $[c]\in H^2(\G_k,*)$ and a finite extension $M$ of $k$
(\resp a place $v$ of $k$),
let $[c]_M\in H^2(\G_M,*)$ be the restriction of $[c]$ to $\G_M$
(\resp $[c]_v\in H^2(\G_{k_v},*)$ be the restriction of $[c]$ to $\G_{k_v}$).
Let $\Br_k$ be the Brauer group of $k$.
We want a small extension of $k$ over which $(A_x,i_x,V_x)$ (or $(A_x,i_x)$)
can be defined.

\begin{prop}[{\cite[Theorem (1.1), p.93]{J}}]
\label{fieldMB}

We can take $(A_x,i_x)$ to be defined over $k$
if and only if
$B\otimes_{\Q}k\cong\M_2(k)$.

\end{prop}

\begin{prop}
\label{fieldM0Bp}

(1)
Suppose $B\otimes_{\Q}k\cong\M_2(k)$.
Further assume $\Aut(x)\ne\{\pm 1\}$ or
$\Aut(x')\not\cong\Z/4\Z$.
Then we can take $(A_x,i_x,V_x)$ to be defined over $k$.

(2)
Assume $\Aut(x)=\{\pm 1\}$.
Then there is a quadratic extension $K$ of $k$
such that we can take $(A_x,i_x,V_x)$ to be defined over $K$.

\end{prop}

\begin{proof}

Let $M$ be a finite extension of $k$.
Then we can take $(A_x,i_x,V_x)$ (\resp $(A_x,i_x)$) to be defined over $M$
if and only if $[c_x]_M=0$ in $H^2(\G_M,\Aut(x))$
(\resp $[c'_x]_M=0$ in $H^2(\G_M,\Aut(x'))$)
(\cf \cite[Main Theorem, p.304 and Final Note, p.336]{DD}).

(1)
In this case we have $[c'_x]=0$ in $H^2(\G_k,\Aut(x'))$
by Proposition \ref{fieldMB}.
Let $\Phi:H^2(\G_k,\Aut(x))\longrightarrow H^2(\G_k,\Aut(x'))$
be the map induced from the inclusion $\Aut(x)\subseteq\Aut(x')$.
If $\Aut(x)\ne\{\pm 1\}$, then $\Aut(x)=\Aut(x')$ and the map $\Phi$
is the identity map (so $\Phi$ is injective in particular).
Therefore $[c_x]=0$.
If $\Aut(x')\not\cong\Z/4\Z$, then $\Aut(x)$ is a direct factor of $\Aut(x')$
as a $\G_k$-module.
Therefore $\Phi$ is injective, and so $[c_x]=0$.

(2)
In this case the action of $\G_k$ on $\Aut(x)=\{\pm 1\}$ is trivial,
and so $\G_k$ acts on $\Aut(x)$ via the mod $2$ cyclotomic character.
Then we have a natural isomorphism
$H^2(\G_k,\Aut(x))\cong\Br_k[2]$.
Let $D$ be the quaternion algebra over $k$ corresponding to $[c_x]$.
We know that there is a quadratic extension $K$ of $k$ such that
$D\otimes_k K\cong\M_2(K)$
(see the proof of Lemma \ref{fieldofdef} below for a detailed explanation).
For such an extension $K$, we have $[c_x]_K=0$ in $H^2(\G_K,\Aut(x))$.

\end{proof}

\begin{lem}
\label{fieldofdef}

Let $K$ be a quadratic extension of $k$.
Assume $\Aut(x)=\{\pm 1\}$.
Then the following two conditions are equivalent.

(1)
We can take $(A_x,i_x,V_x)$ to be defined over $K$.

(2)
For any place $v$ of $k$ satisfying $[c_x]_v\ne 0$,
the tensor product $K\otimes_k k_v$ is a field.

\end{lem}

\begin{proof}

The condition (1) holds if and only if
$[c_x]_K=0$ in $H^2(\G_K,\Aut(x))$.
Let $D$ be the quaternion algebra over $k$ corresponding to $[c_x]$.
We have $[c_x]_K=0$ if and only if
$D\otimes_k K_w\cong\M_2(K_w)$ for any place $w$ of $K$.
The last condition is equivalent to the following:
for any place $v$ of $k$ where $D$ ramifies, 
the tensor product $K\otimes_k k_v$ is a field.

\end{proof}

\begin{rmk}
\label{Kexists}
\rm

We can explicitly construct a quadratic extension $K$ of $k$
in Lemma \ref{fieldofdef} (2) as explained below.
First notice that $k$ has no real places owing to the assumption
$M_0^B(p)(k)\ne\emptyset$.
So, for any infinite place $v$ of $k$,
we have $[c_x]_v=0$.
Let $\mfp_1,\cdots,\mfp_m$, $\mfq_1,\cdots,\mfq_n$ be distinct primes
of $k$.
For simplicity, assume that $2$ is not divisible by any $\mfp_i$;
this is enough for later use.
By the Chinese Remainder Theorem, we can choose an element $t\in\cO_k$
satisfying the following two conditions.

(i) $t\bmod{\mfp_i}$ is not a square in $\cO_k/\mfp_i$ for any $i$.

(ii) $t\bmod{\mfq_j^2}$ is a non-zero element in $\mfq_j/\mfq_j^2$
for any $j$.

\noindent
If we put $K=k(\sqrt{t})$, then each $\mfp_i$ (\resp $\mfq_j$)
 is inert (\resp ramified) in $K/k$.

\end{rmk}
%


\section{Classification of characters (I)}
\label{charI}

We keep the notation in Section \ref{fieldofdefinition}.
Throughout this section,
assume $\Aut(x)=\{\pm 1\}$.
Let $K$ be a quadratic extension of $k$
which satisfies the equivalent conditions in
Lemma \ref{fieldofdef}.
Then $x$ is represented by a triple $(A,i,V)$,
where $(A,i)$ is a QM-abelian surface over $K$ and $V$ is a left
$\cO$-submodule of $A[p](\Kb)$ with $\F_p$-dimension $2$
stable under the action of $\G_K$.
Let
$$\lambda:\G_K\longrightarrow\F_p^{\times}$$
be the character associated to $V$ in (2.2). 
For a prime $\mfl$ of $k$ (\resp $K$), let $I_{\mfl}$ denote the inertia
subgroup of $\G_k$ (\resp $\G_K$) at $\mfl$.

\begin{lem}
\label{lambdaunr}

The character $\lambda^{12}$ is unramified at every prime of
$K$ not dividing $p$.

\end{lem}

\begin{proof}

Let $\mfl$ be a prime of $K$ with residual characteristic $l$.
If $d\ne 6$ or $l\geq 5$,
then $A\otimes_K K_{\mfl}^{unr}$ has good reduction after
a cyclic extension $M/K_{\mfl}^{unr}$ of degree $1,2,3,4$ or $6$,
where $K_{\mfl}^{unr}$ is the maximal unramified extension of $K_{\mfl}$
(\cite[Proposition 3.4 (1), p.101]{J}).
Suppose $d=6$ and $l\leq 3$.
Then $A\otimes_K K_{\mfl}^{unr}$ has good reduction after a
finite Galois extension $M/K_{\mfl}^{unr}$ such that
every element of $\Gal(M/K_{\mfl}^{unr})$ ($\cong I_{\mfl}/\G_M$) is of order $1,2,3,4$ or $6$
(loc.\ cit.).
If $\mfl$ does not divide $p$, then
the action of $\G_M$ (which is a subgroup of $I_{\mfl}$)
on the Tate module $T_pA$ is trivial, and so $\lambda|_{\G_M}=1$.
For any $\sigma\in I_{\mfl}$, we have $\sigma^{12}\in\G_M$
and $\lambda^{12}(\sigma)=\lambda(\sigma^{12})=1$.
Therefore $\lambda^{12}$ is unramified at $\mfl$.

\end{proof}

Let $\lambdaab:\Gab_K\longrightarrow\F_p^{\times}$
be the natural map induced from $\lambda$.
Put
\begin{equation}
\label{phi}
\varphi:=\lambdaab\circ\tr_{K/k}:\G_k\longrightarrow\Gab_K
\longrightarrow\F_p^{\times}, \nonumber
\leqno(5.1)
\end{equation}
where $\tr_{K/k}:\G_k\longrightarrow\Gab_K$ is the transfer map.
Notice that the induced map
$\trab_{K/k}:\Gab_k\longrightarrow\Gab_K$
from $\tr_{K/k}$ corresponds to the natural inclusion
$\kA^{\times}\hookrightarrow\KA^{\times}$
via class field theory
(\cite[Theorem 8 in \S 9 of Chapter XIII, p.276]{We}).

\begin{cor}
\label{phiunr}

The character $\varphi^{12}$ is unramified at every prime of
$k$ not dividing $p$.

\end{cor}

\begin{proof}

Let $\mfq$ be a prime of $k$ not dividing $p$.
The restriction $\varphi^{12}|_{I_{\mfq}}$ can be expressed
by class field theory in the 
following form:
$$\varphi^{12}|_{I_{\mfq}}:
\cO_{k_{\mfq}}^{\times}\longrightarrow
\prod_{\mfq_K |\mfq}K_{\mfq_K}^{\times}
\longrightarrow\F_p^{\times},$$
where $\mfq_K$ is a prime of $K$ above $\mfq$,
the left map is the diagonal map
(where we consider $\cO_{k_{\mfq}}^{\times}$ to be a subgroup
of each $K_{\mfq_K}^{\times}$)
and the right map is induced from $\lambda^{12}$.
By Lemma \ref{lambdaunr}, the right map is trivial on
$\displaystyle\prod_{\mfq_K |\mfq}\cO_{K_{\mfq_K}}^{\times}$.
Therefore $\varphi^{12}|_{I_{\mfq}}$ is trivial.

\end{proof}

\begin{lem}
\label{phi4indep}

The character $\varphi^4$ (and so $\varphi^{12}$)
does not depend on the choice of a quadratic extension $K$ of $k$
which satisfies the equivalent conditions in Lemma \ref{fieldofdef}.

\end{lem}

\begin{proof}

Let $K_1, K_2$ be distinct quadratic extensions of $k$, and assume that
$x$ is represented by a triple $(A_j,i_j,V_j)$ defined over $K_j$
for $j=1,2$.
Then we have two characters
$\lambda_1:\G_{K_1}\longrightarrow\F_p^{\times}$,
$\lambda_2:\G_{K_2}\longrightarrow\F_p^{\times}$
as in (2.2).
In the same way as (5.1),
put $\varphi_j:=\lambdaab_j\circ\tr_{K_j/k}$ for $j=1,2$.
Put $M:=K_1 K_2$.
Then $(A_1,i_1,V_1)\otimes_{K_1}M$ and $(A_2,i_2,V_2)\otimes_{K_2}M$ are isomorphic
over $\kb$.
Since each element of
$H^1(\G_M,\Aut(x))=\Hom(\G_M,\{\pm 1\})$
is trivialized by a quadratic extension of $M$,
there is a quadratic extension $N$ of $M$ such that
$(A_1,i_1,V_1)\otimes_{K_1}M$ and $(A_2,i_2,V_2)\otimes_{K_2}M$
are isomorphic over $N$
(\cf \cite[Theorem 2.2 in \S 2 of Chapter X, p.318]{Si}).
Hence $\lambda_1|_{\G_N}=\lambda_2|_{\G_N}$.
We have the following commutative diagram for $j=1,2$ via class field theory:
\begin{equation*}
\begin{CD}
\kA^{\times}@>\text{$\subseteq$}>> {K_j}_{\A}^{\times}@>\text{$\lambda_j^4$}>> \F_p^{\times} \\
@V\text{$\subseteq$}VV @VVV @V\text{id}VV \\
N_{\A}^{\times}@>\text{$\Norm_{N/K_j}$}>> {K_j}_{\A}^{\times}@>\text{$\lambda_j$}>> \F_p^{\times},
\end{CD}
\end{equation*}
where the middle vertical map is $x\mapsto x^4$.
Then
$\varphi_j^4=(\lambda_j|_{\G_N})^{\text{ab}}
\circ (\tr_{N/k}:\G_k\longrightarrow\Gab_N)$ for $j=1,2$,
where $(\lambda_j|_{\G_N})^{\text{ab}}:\Gab_N\longrightarrow\F_p^{\times}$
is the natural map induced from the restriction $\lambda_j|_{\G_N}$.
Therefore $\varphi_1^4=\varphi_2^4$.

\end{proof}

Let $\theta_p$ denote the mod $p$ cyclotomic character.
By Lemma \ref{lambdaunr} (\resp Corollary \ref{phiunr}),
$\lambda^{12}$ (\resp $\varphi^{12}$) corresponds to a character of the
ideal group $\mfI_K(p)$ (\resp $\mfI_k(p)$) consisting of
fractional ideals of $K$ (\resp $k$) prime to $p$.
By abuse of notation, let denote also by $\lambda^{12}$ (\resp $\varphi^{12}$)
the corresponding character
of $\mfI_K(p)$ (\resp $\mfI_k(p)$).

\begin{lem}
\label{epsilon}

(1)
Let $\cP$ be a prime of $K$ lying over $p$.
Assume $p\geq 5$ and that $\cP$ is unramified in $K/\Q$.
Then $\lambda^{12}|_{I_{\cP}}=\theta_p^b$
for an element $b\in\{0,4,6,8,12\}$.
Further we can take $b\in\{0,6,12\}$ (\resp $b\in\{0,4,8,12\}$)
if $p\not\equiv 2\bmod{3}$ (\resp $p\not\equiv 3\bmod{4}$).

(2)
Assume that $k$ is Galois over $\Q$.
Suppose $p\geq 5$ and that $p$ is unramified in $k$.
Take a prime $\mfp$ of $k$ lying over $p$.
Then there exists an element
$$\varepsilon=\sum_{\sigma\in\Gal(k/\Q)}a_{\sigma}\sigma\in\Z[\Gal(k/\Q)]$$
with $a_{\sigma}\in\{0,8,12,16,24\}$
satisfying the following three conditions.

(i)
$\varphi^{12}(\alpha\cO_k)\equiv\alpha^{\varepsilon}\bmod{\mfp}$
for each element $\alpha\in k^{\times}$ prime to $p$.


(ii)
If $p\not\equiv 2\bmod{3}$ (\resp $p\not\equiv 3\bmod{4}$),
then
$a_{\sigma}\in\{0,12,24\}$ (\resp $a_{\sigma}\in\{0,8,16,24\}$)
for any $\sigma\in\Gal(k/\Q)$.


(iii)
$\varphi^{12}|_{I_{\mfp^{\sigma^{-1}}}}=\theta_p^{a_{\sigma}}$
for any $\sigma\in\Gal(k/\Q)$.

\end{lem}

\begin{proof}

(1)
The abelian surface $A\otimes K_{\cP}$ has good reduction
over a totally ramified extension $M/K_{\cP}$ of degree $e_M=1,2,3,4$ or $6$
since $p\geq 5$ (\cite[Proposition 3.4 (1), p.101]{J}).
We may take $e_M=4$ or $6$.
Since $\cP$ is unramified in $K/\Q$, we have
$\lambda|_{I_{\cP}}=\theta_p^c$ for some $c\in\Z$
(\cite[Proposition 5, p.266 and Proposition 8, p.269]{Se}).
Then $\lambda|_{I_M}=\theta_M^{ce_M}$,
where $I_M$ is the inertia subgroup of $G_M$ and
$\theta_M$ is the fundamental character of level $1$ of $M$ in the sense of
\cite[p.267]{Se}.
For two integers $x,y$ with $y>0$,
let $\langle x\rangle_y$ be the unique integer satisfying
$\langle x\rangle_y\equiv x\bmod{y}$
and
$0\leq \langle x\rangle_y<y$
(\cf \cite[p.150]{Ma}).
We show $0\leq \langle ce_M\rangle_{p-1}\leq e_M$.
By \cite[Corollaire 3.4.4, p.270]{Ray}, we have
$\lambda|_{I_M}=\theta_M^a$
for $0\leq a\leq e_M$.
This combined with $\lambda|_{I_M}=\theta_M^{\langle ce_M\rangle_{p-1}}$
implies that
$\langle ce_M\rangle_{p-1}-a$ is divisible by $p-1$.
If $e_M\geq p-1$, then
$0\leq \langle ce_M\rangle_{p-1}<p-1\leq e_M$.
If $e_M<p-1$, then
$-(p-1)<-e_M\leq \langle ce_M\rangle_{p-1}-a<p-1$,
and so
$\langle ce_M\rangle_{p-1}=a$.
Therefore
$0\leq \langle ce_M\rangle_{p-1}\leq e_M$.

We claim
$\lambda^{12}|_{I_{\cP}}=\theta_p^b$
for an element $b\in\{0,4,6,8,12\}$.

First assume $e_M=4$.
Since $p-1$ is even, we have
$\langle ce_M\rangle_{p-1}=\langle 4c\rangle_{p-1}=0,2$ or $4$.
If $\langle 4c\rangle_{p-1}=0$ or $4$,
then
$\lambda^{12}|_{I_{\cP}}=(\theta_p^{4c})^3=1$ or $\theta_p^{12}$.
If $\langle 4c\rangle_{p-1}=2$,
then
$\lambda^{12}|_{I_{\cP}}=(\theta_p^{4c})^3=\theta_p^6$.
In this case, since $p-1$ divides $4c-2$, we have $p\equiv 3\bmod{4}$.

Next assume $e_M=6$.
Since $p-1$ is even, we have
$\langle ce_M\rangle_{p-1}=\langle 6c\rangle_{p-1}=0,2,4$ or $6$.
If $\langle 6c\rangle_{p-1}=0$ or $6$,
then
$\lambda^{12}|_{I_{\cP}}=(\theta_p^{6c})^2=1$ or $\theta_p^{12}$.
If $\langle 6c\rangle_{p-1}=2$ or $4$,
then
$\lambda^{12}|_{I_{\cP}}=(\theta_p^{6c})^2=\theta_p^4$ or $\theta_p^8$.
In this case, since $p-1$ divides $6c-2$ or $6c-4$, we have $p\equiv 2\bmod{3}$.

(2)
By replacing $K$ if necessary,
we may assume that every prime of $k$ above $p$ is inert in $K/k$
(see the condition (2) in Lemma \ref{fieldofdef}).
Let $\mfp'$ be a prime of $k$ above $p$, and
let $\cP$ be the prime of $K$ above $\mfp'$.
Then the restriction $\varphi^{12}|_{I_{\mfp'}}$ has the
following form via class field theory:
$$\varphi^{12}|_{I_{\mfp'}}:
\cO_{k_{\mfp'}}^{\times}\longrightarrow\cO_{K_{\cP}}^{\times}
\longrightarrow\F_p^{\times},$$
where the left map is the natural inclusion and the right map is
$\Norm_{K_{\cP}/\Q_p}^{-b_{\mfp'}}\bmod{p}$
with $b_{\mfp'}\in\{0,4,6,8,12\}$ as in (1).
Then
$\varphi^{12}|_{I_{\mfp'}}(x)\equiv\Norm_{K_{\cP}/\Q_p}(x)^{-b_{\mfp'}}
=(\Norm_{k_{\mfp'}/\Q_p}\circ\Norm_{K_{\cP}/k_{\mfp'}}(x))^{-b_{\mfp'}}
=\Norm_{k_{\mfp'}/\Q_p}(x)^{-2b_{\mfp'}}\bmod{p}$
for any $x\in\cO_{k_{\mfp'}}^{\times}$.
Hence $\varphi^{12}|_{I_{\mfp'}}=\theta_p^{2b_{\mfp'}}$.
For each element $\alpha\in k^{\times}$ prime to $p$, we have
$$\varphi^{12}(\alpha\cO_k)
=\varphi^{12}(1,\ldots,1,\alpha,\alpha,\ldots)
=\varphi^{12}(\alpha^{-1},\ldots,\alpha^{-1},1,1,\ldots)$$
$$=\prod_{\mfp'|p}\varphi^{12}|_{I_{\mfp'}}(\alpha^{-1})
\equiv\prod_{\mfp'|p}\Norm_{k_{\mfp'}/\Q_p}(\alpha)^{2b_{\mfp'}}\bmod{p}
\equiv\prod_{D_{\mfp}\cdot\sigma\in (D_{\mfp}\backslash\Gal(k/\Q))}
\prod_{\tau\in D_{\mfp}\cdot\sigma}\alpha^{2(b_{\mfp^{\sigma}})\tau^{-1}}\bmod{\mfp}$$
by class field theory again.
Here $(1,\ldots,1,\alpha,\alpha,\ldots)$
(\resp $(\alpha^{-1},\ldots,\alpha^{-1},1,1,\ldots)$)
is the element of $\kA^{\times}$
whose components above $p$ are $1$ and the others $\alpha$
(\resp whose components above $p$ are $\alpha^{-1}$ and the others $1$),
and $D_{\mfp}$ is the decomposition group of $\Gal(k/\Q)$ at $\mfp$.

\end{proof}

For a prime number $q$, put
$$\cFR(q):=\Set{\beta\in\C|
\beta^2+a\beta+q=0 \text{ for some integer $a\in\Z$ with $|a|\leq 2\sqrt{q}$}}.$$
Notice that $|a|\leq 2\sqrt{q}$ implies $|a|<2\sqrt{q}$
since $2\sqrt{q}$ is not a rational number.
Notice also that for any $\beta\in\cFR(q)$,
we have $\beta\not\in\R$ and $|\beta|=\sqrt{q}$.
Recall that $h_k$ is the class number of $k$.

\begin{lem}
\label{type23epsilon}

Under the situation in Lemma \ref{epsilon} (2), take a prime number
q different from $p$. Assume that $q$ splits completely in $k$,
and take a prime $\mfq$ of $k$ lying over $q$.
Let $\alpha\in\cO_k\setminus\{0\}$ be an element such that
$\mfq^{h_k}=\alpha\cO_k$.
If $\alpha^{\varepsilon}=\beta^{24h_k}$
for an  element $\beta\in\cFR(q)$,
then $\varepsilon$ is of one of the following types.

Type 2:
$\varepsilon=\displaystyle\sum_{\sigma\in\Gal(k/\Q)}12\sigma$
and $p\equiv 3\bmod{4}$.

Type 3:
$k$ contains $\Q(\beta)$, and
$\varepsilon=\displaystyle\sum_{\sigma\in\Gal(k/\Q(\beta))}24\sigma$
or $\displaystyle\sum_{\sigma\not\in\Gal(k/\Q(\beta))}24\sigma$.

\end{lem}

\begin{proof}

Put $L:=\Q(\beta)$. Then $L$ is an imaginary quadratic field.
We have the inclusion of fields 
$\Q\subseteq\Q(\alpha^{\varepsilon})\subseteq L$.
Since $[L:\Q]=2$, we have the following two cases:
(i) $\alpha^{\varepsilon}\in\Q$,
(ii) $\Q(\alpha^{\varepsilon})=L$.

In case (i), we have
$\alpha^{\varepsilon}=\pm q^{12h_k}$, because
$|\alpha^{\varepsilon}|=|\beta^{24h_k}|=q^{12h_k}$ and
$\alpha^{\varepsilon}\in\Q$.
As $q^{12h_k}\cO_k=\alpha^{\varepsilon}\cO_k=\mfq^{h_k\varepsilon}$,
we have
$\mfq^{\varepsilon}=q^{12}\cO_k$.
Since $q$ splits completely in $k$
and
$q\cO_k=\prod_{\sigma\in\Gal(k/\Q)}\mfq^{\sigma}$,
we get
$\varepsilon=\sum_{\sigma\in\Gal(k/\Q)}12\sigma$.
In this case we have $p\equiv 3\bmod{4}$ by the condition (ii)
in Lemma \ref{epsilon} (2).

In case (ii), we have
$k\supseteq\Q(\alpha^{\varepsilon})=L$.
Consider the ideal
$\alpha^{\varepsilon}\cO_k=\mfq^{h_k\varepsilon}
=\prod_{\sigma\in\Gal(k/\Q)}\mfq^{h_k a_{\sigma}\sigma}$.
This is $\Gal(k/L)$-invariant since $\alpha^{\varepsilon}\in L$.
So, for any $\tau\in\Gal(k/L)$, we have
$\prod_{\sigma\in\Gal(k/\Q)}\mfq^{h_k a_{\sigma}\sigma}
=\prod_{\sigma\in\Gal(k/\Q)}\mfq^{h_k a_{\sigma}\sigma\tau}$.
Since 
$\mfq^{\sigma}\ne\mfq^{\sigma'}$ for $\sigma\ne\sigma'$,
we have
$$\varepsilon=\sum_{\sigma\in\Gal(k/L)}a\sigma
+\sum_{\sigma\not\in\Gal(k/L)}b\sigma$$
where $a,b\in\{0,8,12,16,24\}$.
By the assumption we have
$\beta^{24h_k}\cO_k=\alpha^{\varepsilon}\cO_k
=\mfq^{h_k\varepsilon}$,
and so
$$\beta^{24}\cO_k=\mfq^{\varepsilon}
=\prod_{\sigma\in\Gal(k/L)}\mfq^{a\sigma}
\prod_{\sigma\not\in\Gal(k/L)}\mfq^{b\sigma}.$$
Taking the intersection with $\cO_L$, we have
$\beta^{24}\cO_L=q_L^a {q'_L}^b$
where $q_L,q'_L$ are the two distinct primes of $L$ above $q$
with $q_L=\mfq\cap\cO_L$.
By taking 
the absolute norms, we see
$q^{24}=q^{a+b}$ and so $a+b=24$.
Applying prime ideal decomposition to
$q_L^a {q'_L}^b=(\beta\cO_L)^{24}$,
we get $(a,b)=(24,0)$ or $(0,24)$.
Therefore
$\varepsilon=\sum_{\sigma\in\Gal(k/L)}24\sigma$
or $\sum_{\sigma\not\in\Gal(k/L)}24\sigma$.
%

\end{proof}

Let $\cM$ be the set of prime numbers $q$ such that $q$ splits
completely in $k$
and $q$ does not divide $6h_k$.
Let $\cN$ be the set of primes $\mfq$ of $k$ such that
$\mfq$ divides some prime number $q\in \cM$.
Let $\mfI_k$ be the ideal group of $k$, and let $\mfP_k$
be the subgroup of $\mfI_k$ consisting of principal ideals.
Take a finite subset $\emptyset\ne \cS\subseteq \cN$ such that
$\cS$ generates the ideal class group $Cl_k=\mfI_k/\mfP_k$.
For each prime $\mfq\in \cS$, fix an element $\alpha_{\mfq}\in\cO_k\setminus\{0\}$
satisfying $\mfq^{h_k}=\alpha_{\mfq}\cO_k$.
%



For $\mfq\in\cS$,
put $\N(\mfq)=\sharp(\cO_k/\mfq)$.
Then $\N(\mfq)$ is a prime number.
Define the sets

\noindent
$\cM_1(k):=$
$$\Set{(\mfq,\varepsilon_0,\beta_{\mfq})|
\mfq\in \cS,\ \varepsilon_0=\sum_{\sigma\in\Gal(k/\Q)}a_{\sigma}\sigma
\text{ with $a_{\sigma}\in\{0,8,12,16,24 \}$},\ 
\beta_{\mfq}\in\cFR(\N(\mfq))},$$

\noindent
$\cM_2(k):=\Set{\Norm_{k(\beta_{\mfq})/\Q}(\alpha_{\mfq}^{\varepsilon_0}-\beta_{\mfq}^{24h_k})\in\Z|
(\mfq,\varepsilon_0,\beta_{\mfq})\in\cM_1(k)}\setminus\{0\}$,

\noindent
$\cN_0(k):=\Set{\text{$l$ : prime number}|\text{$l$ divides some integer $m\in\cM_2(k)$}}$,

\noindent
$\cT(k):=\Set{\text{$l'$ : prime number}|\text{$l'$ is divisible
by some prime $\mfq'\in \cS$}}
\cup\{2,3\}$,

\noindent
$\cN_1(k):=\cN_0(k)\cup\cT(k)\cup\Ram(k)$.

\noindent
Notice that all of $\cM_1(k)$, $\cM_2(k)$, $\cN_0(k)$, $\cT(k)$,
$\cN_1(k)$ are finite.

\begin{thm}
\label{type23phi}

Assume that $k$ is Galois over $\Q$.
If $p\not\in\cN_1(k)$
(and if $p$ does not divide $d$),
then the character
$\varphi:\G_k\longrightarrow \F_p^{\times}$
is of one of the following types.

Type 2:
$\varphi^{12}=\theta_p^{12}$ and $p\equiv 3\bmod{4}$.

Type 3:
There is an imaginary quadratic field $L$ satisfying the following
two conditions.


\noindent
(a)
The Hilbert class field $H_L$ of $L$ is contained in $k$.

\noindent
(b)
There is a prime $\mfp_L$ of $L$ lying over $p$
such that
$\varphi^{12}(\mfa)\equiv\delta^{24}\bmod{\mfp_L}$ holds
for any fractional ideal $\mfa$ of $k$ prime to $p$.
Here $\delta$ is any element of $L$ such that
$\Norm_{k/L}(\mfa)=\delta\cO_L$.

\end{thm}

\begin{proof}

By replacing $K$ if necessary, we may assume
that every prime $\mfq\in \cS$ is ramified in $K/k$
(see the condition (2) in Lemma \ref{fieldofdef}).

Suppose $p\not\in\cT(k)\cup\Ram(k)$.
Take any prime $\mfq\in \cS$.
Let $q$ be the residual characteristic of $\mfq$, and
let $\mfq_K$ be the unique prime of $K$ above $\mfq$.
Then $p\ne q$.
Since $q\geq 5$, the abelian surface $A\otimes_K K_{\mfq_K}$ over $K_{\mfq_K}$
(which corresponds to $x\otimes K_{\mfq_K}$)
has good reduction after
a totally ramified finite extension $M(\mfq)/K_{\mfq_K}$
(\cite[Proposition 3.4 (1), p.101]{J}).
Then 
$\lambda^{12}(\mfq_K)=\lambda^{12}(\Frob_{\mfq_K})
=\lambda^{12}|_{\G_{M(\mfq)}}(\Frob_{M(\mfq)})=(\lambda|_{\G_{M(\mfq)}}(\Frob_{M(\mfq)}))^{12}$,
where $\Frob_{\mfq_K}$ (\resp $\Frob_{M(\mfq)}$) is the arithmetic Frobenius
of $\G_K$ at $\mfq_K$
(\resp the arithmetic Frobenius of $\G_{M(\mfq)}$).
Here note that $\lambda|_{\G_{M(\mfq)}}(\Frob_{M(\mfq)})$ is well-defined
since $\lambda|_{\G_{M(\mfq)}}$ is unramified.
There is a Frobenius eigenvalue $\beta_{\mfq}$ on some QM-abelian surface by $\cO$
over the \textit{prime field} $\F_q$
such that $\lambda|_{\G_{M(\mfq)}}(\Frob_{M(\mfq)})\equiv\beta_{\mfq}$
modulo a prime of $\Q(\beta_{\mfq})$ above $p$,
since $q$ splits completely in $k$
and $\mfq_K/\mfq$ is ramified.
Then we know $\beta_{\mfq}\in\cFR(q)$ by \cite[p.97]{J}.
Choose one prime $\mfp$ of $k$ above $p$.
Take a prime $\mfp_1$ of $k(\beta_{\mfq}\,|\,\mfq\in\cS)$ above $\mfp$.
Let $\mfp_2$ be the prime of $\Q(\beta_{\mfq}\,|\,\mfq\in\cS)$ below $\mfp_1$.
Replacing each $\beta_{\mfq}$ by its complex conjugate if necessary,
we may assume
$\lambda|_{\G_{M(\mfq)}}(\Frob_{M(\mfq)})\equiv\beta_{\mfq}\bmod{\mfp_2}$
for any $\mfq\in\cS$.
Then $\lambda^{12}(\mfq_K)\equiv\beta_{\mfq}^{12}\bmod{\mfp_2}$.

Applying Lemma \ref{epsilon} (2) to $\mfp$, we find an element
$\varepsilon=\sum_{\sigma\in\Gal(k/\Q)}a_{\sigma}\sigma\in\Z[\Gal(k/\Q)]$
with $a_{\sigma}\in\{0,8,12,16,24\}$,
which satisfies
$\varphi^{12}(\gamma\cO_k)\equiv\gamma^{\varepsilon}\bmod{\mfp}$
for any $\gamma\in k^{\times}$ prime to $p$
together with the condition (ii) in Lemma \ref{epsilon} (2).
In particular, for any prime $\mfq\in\cS$, we have
$\alpha_{\mfq}^{\varepsilon}\equiv\varphi^{12}(\alpha_{\mfq}\cO_k)
=\varphi^{12}(\mfq^{h_k})
=\lambda^{24h_k}(\mfq_K)
\equiv\beta_{\mfq}^{24h_k}\bmod{\mfp_1}$.
The second equality holds since $\mfq_K/\mfq$ is ramified.
Then $\alpha_{\mfq}^{\varepsilon}-\beta_{\mfq}^{24h_k}$
belongs to the prime of $k(\beta_{\mfq})$ below $\mfp_1$.
Therefore $p$ divides the rational integer
$\Norm_{k(\beta_{\mfq})/\Q}(\alpha_{\mfq}^{\varepsilon}-\beta_{\mfq}^{24h_k})$
for any prime $\mfq\in\cS$.


Suppose $p\not\in\cN_1(k)$.
Then, for any prime $\mfq\in \cS$, 
we have $\Norm_{k(\beta_{\mfq})/\Q}(\alpha_{\mfq}^{\varepsilon}-\beta_{\mfq}^{24h_k})=0$,
and so $\alpha_{\mfq}^{\varepsilon}=\beta_{\mfq}^{24h_k}$.
%
Choose one prime $\mfq_0\in\cS$. Applying Lemma \ref{type23epsilon}
to $\mfq_0$, we know that $\varepsilon$ is of type 2 or 3.

%


First assume that $\varepsilon$ is of type 2.
For any prime $\mfq\in \cS$, 
we have $\beta_{\mfq}^{24h_k}=\alpha_{\mfq}^{\varepsilon}=\Norm_{k/\Q}(\alpha_{\mfq})^{12}
\in\Set{t\in\Q|t>0}$.
We also have
$|\beta_{\mfq}^{24h_k}|=q^{12h_k}$, where $q$ is the residual characteristic
of $\mfq$.
Then $\beta_{\mfq}^{24h_k}=q^{12h_k}$,
and so $\beta_{\mfq}=\zeta\sqrt{-q}$ for some $24h_k$-th root $\zeta$ of unity.
We claim $\beta_{\mfq}=\pm\sqrt{-q}$.
In the following proof of this claim, we write $\beta=\beta_{\mfq}$ for simplicity.

(I) Case $\Q(\beta)\ne\Q(\sqrt{-q})$.
Since $\Q(\zeta,\sqrt{-q})=\Q(\beta,\sqrt{-q})$,
we have
$[\Q(\zeta,\sqrt{-q}):\Q]=4$.
Furthermore we consider the following three cases:
(I-i)
$\Q(\zeta)=\Q$,
(I-ii)
$[\Q(\zeta):\Q]=2$ and
(I-iii)
$\Q(\zeta)=\Q(\zeta,\sqrt{-q})$.

Case (I-i).
We have $\zeta=\pm 1$ and $\beta=\pm\sqrt{-q}$.
But this is impossible because $\Q(\beta)\ne\Q(\sqrt{-q})$.

Case (I-ii).
We have $\ord\zeta=3,4$ or $6$.
If $\ord\zeta=4$, then $\beta^2=\zeta^2(-q)=q>0$ and $\beta\in\R$.
This is a contradiction.
If $\ord\zeta=3$, then $\beta^2=\zeta^2(-q)=-\zeta^{-1}q$.
In this case $\Q(\beta)\supseteq\Q(\zeta)$, which implies
$\Q(\beta)=\Q(\zeta)$.
Then $\Q(\zeta)\ni\beta=\zeta\sqrt{-q}$, and so
$\Q(\zeta)\ni\sqrt{-q}$.
Hence $\Q(\zeta,\sqrt{-q})=\Q(\zeta)$.
This contradicts $[\Q(\zeta,\sqrt{-q}):\Q]=4$.
If $\ord\zeta=6$, then $\Q(\zeta)=\Q(\zeta^2)$ and
$\beta^2=-\zeta^2 q$.
In this case $\Q(\beta)\supseteq\Q(\zeta^2)=\Q(\zeta)$.
This leads to a contradiction just as in the case where $\ord\zeta=3$.

Case (I-iii).
We have $\Q(\zeta)\supseteq\Q(\sqrt{-q})$.
But the prime number $q$ is unramified in $\Q(\zeta)$ since $q$
does not divide $6h_k$.
This is a contradiction.

(II) Case $\Q(\beta)=\Q(\sqrt{-q})$.
In this case $\zeta\sqrt{-q}=\beta\in\Q(\sqrt{-q})$.
Then $\zeta\in\Q(\sqrt{-q})$ and
$\Q(\zeta)\subseteq\Q(\sqrt{-q})$.
Furthermore we consider the following two cases:
(II-i) $\Q(\zeta)=\Q$ and (II-ii) $\Q(\zeta)=\Q(\sqrt{-q})$.

Case (II-i).
We have $\zeta=\pm 1$ and $\beta=\pm\sqrt{-q}$, as required.

Case (II-ii).
Since $[\Q(\zeta):\Q]=2$, we have $\ord\zeta=3,4$ or $6$.
This can occur only when $q=3$.
But $q\ne 3$ since $q$ does not divide $6h_k$.


Therefore we conclude $\beta=\pm\sqrt{-q}$.
Then
$\varphi^{12}(\Frob_{\mfq})=\varphi^{12}(\mfq)
=\lambda^{24}(\mfq_K)\equiv\beta_{\mfq}^{24}=q^{12}
=\N(\mfq)^{12}\equiv\theta_p(\Frob_{\mfq})^{12}\bmod{p}$,
where $\Frob_{\mfq}$ is the arithmetic Frobenius of $\G_k$ at $\mfq$.
Combining this with
$\varphi^{12}(\gamma\cO_k)\equiv\Norm_{k/\Q}(\gamma)^{12}\bmod{p}$
for any $\gamma\in k^{\times}$ prime to $p$,
we conclude $\varphi^{12}=\theta_p^{12}$.
%

%
Next assume that $\varepsilon$ is of type 3 (for $\mfq_0$).
In this case $k$ contains the imaginary quadratic field $L=\Q(\beta_{\mfq_0})$,
and
$\varepsilon=\sum_{\sigma\in\Gal(k/L)}24\sigma$
or $\sum_{\sigma\not\in\Gal(k/L)}24\sigma$.
Applying Lemma \ref{type23epsilon} to each prime $\mfq\in\cS$,
we also have $k\supseteq\Q(\beta_{\mfq})$ and
$\varepsilon=\sum_{\sigma\in\Gal(k/\Q(\beta_{\mfq}))}24\sigma$
or $\sum_{\sigma\not\in\Gal(k/\Q(\beta_{\mfq}))}24\sigma$.
Then $\Q(\beta_{\mfq})=L=\Q(\beta_{\mfq_0})$,
which does not depend on $\mfq$.
In this case note that $\mfp_2$ is a prime of $L=\Q(\beta_{\mfq}\,|\,\mfq\in\cS)$.

Case (A) : $\varepsilon=\sum_{\sigma\in\Gal(k/L)}24\sigma$.
For any prime $\mfq\in\cS$, we have
$\Norm_{k/L}(\mfq)^{24h_k}=\Norm_{k/L}(\alpha_{\mfq})^{24}\cO_L
=\alpha_{\mfq}^{\varepsilon}\cO_L=\beta_{\mfq}^{24h_k}\cO_L=(\beta_{\mfq}\cO_L)^{24h_k}$
as ideals of $\cO_L$.
Then $\Norm_{k/L}(\mfq)=\beta_{\mfq}\cO_L$, which is a principal ideal.
Hence $\Norm_{k/L}(\mfI_k)\subseteq \mfP_L$,
where $\mfP_L$ is the principal ideal group of $L$.
Therefore we get $k\supseteq H_L$.
For any prime $\mfq\in\cS$, 
we have
$\varphi^{12}(\mfq)=\lambda^{24}(\mfq_K)\equiv\beta_{\mfq}^{24}\bmod{\mfp_2}$.
%
%
We have also seen $\Norm_{k/L}(\mfq)=\beta_{\mfq}\cO_L$.
Since $k=k(\beta_{\mfq}\,|\,\mfq\in\cS)$, we have $\mfp=\mfp_1$, which divides $\mfp_2$.
For any $\gamma\in k^{\times}$ prime to $p$, we have
$\varphi^{12}(\gamma\cO_k)
\equiv\Norm_{k/L}(\gamma)^{24}\bmod{\mfp_2}$.
Then we get
$\varphi^{12}(\mfa)\equiv\delta^{24}\bmod{\mfp_2}$
for any fractional ideal $\mfa$ of $k$ prime to $p$,
where $\delta\in L$ is a \textit{certain} element such that
$\Norm_{k/L}(\mfa)=\delta\cO_L$.
But the power $\delta^{24}$ is uniquely determined by $\mfa$
since $\delta$ is unique up to multiplication by an element
of $\cO_L^{\times}$ and $\sharp\cO_L^{\times}=2,4$ or $6$.
Then we have done by taking $\mfp_L=\mfp_2$.

Case (B) : $\varepsilon=\sum_{\sigma\not\in\Gal(k/L)}24\sigma$.
We have
$\Norm_{k/\Q}(\mfq)^{24h_k}\Norm_{k/L}(\mfq)^{-24h_k}
=\Norm_{k/\Q}(\alpha_{\mfq})^{24}\Norm_{k/L}(\alpha_{\mfq})^{-24}\cO_L
=\alpha_{\mfq}^{\varepsilon}\cO_L=\beta_{\mfq}^{24h_k}\cO_L=(\beta_{\mfq}\cO_L)^{24h_k}$
as ideals of $\cO_L$.
Then
$\Norm_{k/\Q}(\mfq)\Norm_{k/L}(\mfq)^{-1}=\beta_{\mfq}\cO_L$.
Applying the non-trivial element $c$ of $\Gal(L/\Q)$
we get
$\Norm_{k/L}(\mfq)=\beta_{\mfq}^c\cO_L$,
which is a principal ideal.
We also have $\varphi^{12}(\mfq)=\lambda^{24}(\mfq_K)\equiv\beta_{\mfq}^{24}\bmod{\mfp_2}$.
Then
$\varphi^{12}(\mfq)\equiv (\beta_{\mfq}^c)^{24}\bmod{\mfp_2^c}$.
%
For any $\gamma\in k^{\times}$ prime to $p$, we have
$\varphi^{12}(\gamma\cO_k)
\equiv\Norm_{k/\Q}(\gamma)^{24}\Norm_{k/L}(\gamma)^{-24}\bmod{\mfp_2}$,
and then, applying $c$, we obtain
$\varphi^{12}(\gamma\cO_k)\equiv\Norm_{k/L}(\gamma)^{24}\bmod{\mfp_2^c}$.
Then we get
$\varphi^{12}(\mfa)\equiv\delta^{24}\bmod{\mfp_2^c}$
for any fractional ideal $\mfa$ of $k$ prime to $p$,
where $\delta\in L$ is a certain element (and so any element) such that
$\Norm_{k/L}(\mfa)=\delta\cO_L$.
We have done by taking $\mfp_L=\mfp_2^c$.

\end{proof}

\begin{rmk}
\label{mimictype}
In Lemma \ref{type23epsilon} and Theorem \ref{type23phi},
we mimic the notations ``Type 2", ``Type 3" in
\cite[Lemma 2 and Theorem 1, p.333]{Mo}.
We have no need to deal with ``Type 1" since
every QM-abelian surface has potentially good reduction everywhere
(or a Shimura curve has no cusps).
\end{rmk}

\begin{rmk}
\label{rmktype3}
In Type 3 of Theorem \ref{type23phi}, the element $\delta$
always exists since $\Norm_{k/L}(\mfI_k)\subseteq \mfP_L$.

\end{rmk}

Let $\cN_2(k)$ be the set of prime numbers $l$
satisfying the following two conditions.

(i)
The number $-l$ is a discriminant of some quadratic field.

(ii)
For any prime number $q$ satisfying $3<q<\frac{l}{4}$, if $q$ splits completely in $k$
then $q$ does not split in $\Q(\sqrt{-l})$.

\begin{thm}[{\cite[Theorem A, p.160]{Ma}}]
\label{goldfeld}

If $k$ is a quadratic field, then
the set $\cN_2(k)$ is finite.
Further the upper bound of $\cN_2(k)$ is effectively estimated
except at most one prime number.

\end{thm}

From now to the end of this section, assume that $k$ is Galois over $\Q$.

\begin{prop}
\label{type2}

Suppose $p\geq 11$, $p\ne 13$ and $p\not\in\cN_1(k)$.
If $\varphi$ is of type 2, then $p\in\cN_2(k)$.
%

\end{prop}

We give two lemmas to show Proposition \ref{type2}.

\begin{lem}
\label{type2lambda}

Suppose $p\geq 11$, $p\ne 13$ and $p\not\in\cN_1(k)$.
Further assume the following two conditions.

(a)
Every prime $\mfp$ of $k$ above $p$ is inert in $K/k$.

(b)
Every prime $\mfq\in\cS$ is ramified in $K/k$.

\noindent
If $\varphi$ is of type 2, then we have the following.

(i) The character $\lambda^{12}\theta_p^{-6}:\G_K\longrightarrow\F_p^{\times}$
is unramified everywhere.

(ii) The map $Cl_K\longrightarrow\F_p^{\times}$ induced from $\lambda^{12}\theta_p^{-6}$
is trivial on
$C_{K/k}:=\im(Cl_k\longrightarrow Cl_K)$,
where $Cl_K$ is the ideal class group of $K$
and
$Cl_k\longrightarrow Cl_K$
is the map defined by
$[\mfa]\longmapsto[\mfa\cO_K]$.

\end{lem}

\begin{proof}

(i)
We know that $\lambda^{12}$ is unramified outside $p$ by Lemma
\ref{lambdaunr}.
Then $\lambda^{12}\theta_p^{-6}$ is also unramified outside $p$.
Take a prime $\cP$ of $K$ above $p$.
We show that $\lambda^{12}\theta_p^{-6}|_{I_{\cP}}$ is trivial.
Put $\mfp:=\cP\cap\cO_k$.
By Lemma \ref{epsilon} (1), we have
$\lambda^{12}|_{I_{\cP}}=\theta_p^a$
for an element $a\in\{0,4,6,8,12\}$.
Since $\cP/\mfp$ is inert, we have the following commutative
diagram:
\begin{equation*}
\begin{CD}
\Gab_{k_{\mfp}}@>\text{$\tr$}>>\Gab_{K_{\cP}}
@>\text{$(\lambda^{12}|_{\G_{K_{\cP}}})^{\text{ab}}$}>>\F_p^{\times} \\
@AAA @AAA @AA\text{$id$}A \\
\cO_{k_{\mfp}}^{\times}@>>\text{$\subseteq$}>\cO_{K_{\cP}}^{\times}
@>>\text{$\Norm_{K_{\cP}/\Q_p}^{-a}\bmod{p}$}>\F_p^{\times},
\end{CD}
\end{equation*}
where the left and middle vertical maps are the reciprocity maps
of local class field theory,
$\tr$ is the natural map induced from the transfer map
$\G_{k_{\mfp}}\longrightarrow\Gab_{K_{\cP}}$
and $(\lambda^{12}|_{\G_{K_{\cP}}})^{\text{ab}}$
is the natural map induced from
$\lambda^{12}|_{\G_{K_{\cP}}}$.
For any $x\in\cO_{k_{\mfp}}^{\times}$, we have
$\varphi^{12}(x)\equiv\Norm_{K_{\cP}/\Q_p}(x)^{-a}
=(\Norm_{k_{\mfp}/\Q_p}\circ\Norm_{K_{\cP}/k_{\mfp}}(x))^{-a}
=\Norm_{k_{\mfp}/\Q_p}(x)^{-2a}
\equiv\theta_p(x)^{2a}\bmod{p}$.
Since $\varphi^{12}=\theta_p^{12}$, we have
$\theta_p^{12-2a}(x)\equiv 1\bmod{p}$
for any $x\in\cO_{k_{\mfp}}^{\times}$.
As $p$ is unramified in $k$, we know
$\theta_p(\cO_{k_{\mfp}}^{\times})=\F_p^{\times}$.
Then $2a\equiv 12\bmod{p-1}$.
Since $p\geq 11$ and $p\ne 13$, we have $a=6$.
Hence $\lambda^{12}|_{I_{\cP}}=\theta_p^6$,
and so $\lambda^{12}\theta_p^{-6}|_{I_{\cP}}$ is trivial.

(ii)
Take any prime $\mfq\in\cS$.
Let $q$ be the residual characteristic of $\mfq$, and let
$\mfq_K$ be the unique prime of $K$ above $\mfq$.
In the proof of Theorem \ref{type23phi}, we have seen that
the element $\varepsilon$ in Lemma \ref{epsilon} (2)
is of type 2 or 3.
But, because of the hypothesis $p\geq 11$, $p\ne 13$
and the condition (iii) in Lemma \ref{epsilon} (2),
the element $\varepsilon$ cannot be of type 3.
We also have seen (in the proof of Theorem \ref{type23phi})
$\lambda^{12}(\mfq_K)\equiv\beta^{12}\bmod{p}$
where
$\beta=\pm\sqrt{-q}$
in the case of type 2.
Then
$\lambda^{12}(\mfq\cO_K)
=\lambda^{12}(\mfq_K^2)
=\lambda^{24}(\mfq_K)
\equiv \beta^{24}
=q^{12}
\equiv\theta_p^{12}(\mfq_K)
=\theta_p^6(\mfq\cO_K)\bmod{p}$,
as required.

\end{proof}

\begin{lem}
\label{q/p-1}

Suppose $p\geq 11$, $p\ne 13$ and $p\not\in\cN_1(k)$.
Assume that $\varphi$ is of type 2.
Let $q<\frac{p}{4}$ be a prime number
which splits completely in $k$.
Then we have
$\left(\frac{q}{p}\right)=-1$.

\end{lem}

\begin{proof}

Assume otherwise i.e. $\left(\frac{q}{p}\right)=1$.
Since $\left(\frac{q}{p}\right)\equiv q^{\frac{p-1}{2}}\bmod{p}$,
we have $q^{\frac{p-1}{2}}\equiv 1\bmod{p}$.

Take a prime $\mfq$ of $k$ above $q$.
By replacing $K$ if necessary,
we may assume the conditions (a), (b) in Lemma \ref{type2lambda}
and that $\mfq$ is ramified in $K/k$.
Let $\mfq_K$ be the unique prime of $K$ above $\mfq$.
The abelian variety $A\otimes_K K_{\mfq_K}$ has good reduction
after a totally ramified finite extension $M/K_{\mfq_K}$.
Then $\lambda(\Frob_M)\equiv\beta$
modulo a prime $\mfp_0$ of $\Q(\beta)$ above $p$
for a Frobenius eigenvalue
$\beta$ of some QM-abelian surface over the prime field $\F_q$,
where $\Frob_M$ is the arithmetic Frobenius of $\G_M$
($\subseteq\G_{K_{\mfq_K}}$).
We also have $\lambda^{-1}\theta_p(\Frob_M)\equiv\betab\bmod{\mfp_0}$,
where $\betab$ is the complex conjugate of $\beta$.
Put $\psi:=\lambda\theta_p^{-\frac{p+1}{4}}$.
Then
$\psi^{12}=\lambda^{12}\theta_p^{-3(p+1)}
=\lambda^{12}\theta_p^{-6}$.
By Lemma \ref{type2lambda} (ii),
we have
$1=\lambda^{12}(\mfq\cO_K)\theta_p^{-6}(\mfq\cO_K)
=\psi^{12}(\mfq\cO_K)=\psi^{24}(\mfq_K)
=\psi^{24}(\Frob_M)=\psi(\Frob_M)^{24}$.
Here, note that $\psi(\Frob_M)$ is well-defined and
that the fourth equality holds because the extension
$M/K_{\mfq_K}$ is totally ramified.
Since $\F_p^{\times}$ is a cyclic group of order $p-1$ and
$p-1\equiv 2\bmod{4}$, we have $\psi(\Frob_M)^6=1$.
Then
$\psi(\Frob_M)^2+\psi(\Frob_M)^{-2}\equiv-1$ or $2\bmod{p}$.
We also have
$\beta^2+\betab^2\equiv
\psi(\Frob_M)^2\theta_p(\Frob_M)^{\frac{p+1}{2}}
+\psi(\Frob_M)^{-2}\theta_p(\Frob_M)^{\frac{-p+3}{2}}
=\theta_p(\Frob_M)^{\frac{p+1}{2}}
(\psi(\Frob_M)^2+\psi(\Frob_M)^{-2})
=q^{\frac{p+1}{2}}(\psi(\Frob_M)^2+\psi(\Frob_M)^{-2})\bmod{p}$.
Since $q^{\frac{p-1}{2}}\equiv 1\bmod{p}$, we have
$q^{\frac{p+1}{2}}\equiv q\bmod{p}$.
Then $\beta^2+\betab^2\equiv -q$ or $2q\bmod{p}$,
and so $(\beta+\betab)^2\equiv q$ or $4q\bmod{p}$.

Since $\beta\in\cFR(q)$, we have $|\beta+\betab|<2\sqrt{q}$.
Since $q<\frac{p}{4}$, we have
$|(\beta+\betab)^2-q|<3q<p$ and
$|(\beta+\betab)^2-4q|\leq 4q<p$.
Then $(\beta+\betab)^2\equiv q$ or $4q\bmod{p}$ implies
$(\beta+\betab)^2=q$ or $4q$.
Since $\beta+\betab\in\Z$, we have a contradiction.
Therefore we conclude
$\left(\frac{q}{p}\right)=-1$.

\end{proof}

\noindent
(Proof of Proposition \ref{type2})

Assume that $\varphi$ is of type 2.
In this case $p\equiv 3\bmod{4}$, and so
$-p$ is the discriminant of $\Q(\sqrt{-p})$.
For any $q<\frac{p}{4}$ that splits completely in $k$, we have seen
$\left(\frac{q}{p}\right)=-1$ in Lemma \ref{q/p-1}.
If $q\ne 2$, this implies
$\left(\frac{-p}{q}\right)=\left(\frac{-1}{q}\right)\left(\frac{p}{q}\right)
=\left(\frac{-1}{q}\right)\left(\frac{q}{p}\right)(-1)^{\frac{(p-1)(q-1)}{4}}
=\left(\frac{q}{p}\right)\left(\frac{-1}{q}\right)(-1)^{\frac{q-1}{2}}
=\left(\frac{q}{p}\right)=-1$,
and so $q$ is inert in $\Q(\sqrt{-p})$.
Then $p$ belongs to the set $\cN_2(k)$.


\qed







\section{Classification of characters (II)}
\label{charII}

Let $k$ be a number field, and
let $(A,i)$ be a QM-abelian surface by $\cO$ over $k$.
For a prime number $p$ not dividing $d$,
assume that the representation $\rhob_{A,p}$ in (\ref{rhobar}) is reducible.
Then there is a 1-dimensional sub-representation of $\rhob_{A,p}$;
let $\nu$ be its associated character.
In this case notice that there is a left $\cO$-submodule $V$
of $A[p](\kb)$ with $\F_p$-dimension $2$ on which $\G_k$ acts by $\nu$,
and so the triple $(A,i,V)$ determines a point of $M_0^B(p)(k)$.
%
%
%
We can classify $\nu$
by the same method as in the last section.
%
Below we give the statements of lemmas, a theorem and a proposition,
but omit the proofs since they are almost the same as (or simpler than) in the last section.

\begin{lem}
\label{nuunr}

The character $\nu^{12}$ is unramified at every prime of
$k$ not dividing $p$.

\end{lem}

\begin{lem}
\label{epsilon'}

(1)
Let $\mfp$ be a prime of $k$ lying over $p$.
Assume $p\geq 5$ and that $\mfp$ is unramified in $k/\Q$.
Then $\nu^{12}|_{I_{\mfp}}=\theta_p^b$
for an element $b\in\{0,4,6,8,12\}$.
Further we can take $b\in\{0,6,12\}$ (\resp $b\in\{0,4,8,12\}$)
if $p\not\equiv 2\bmod{3}$ (\resp $p\not\equiv 3\bmod{4}$).

(2)
Assume that $k$ is Galois over $\Q$.
Suppose $p\geq 5$ and that $p$ is unramified in $k$.
Take a prime $\mfp$ of $k$ lying over $p$.
Then there exists an element
$$\varepsilon'=\sum_{\sigma\in\Gal(k/\Q)}a'_{\sigma}\sigma\in\Z[\Gal(k/\Q)]$$
with $a'_{\sigma}\in\{0,4,6,8,12\}$
satisfying the following three conditions.

(i)
$\nu^{12}(\alpha\cO_k)\equiv\alpha^{\varepsilon'}\bmod{\mfp}$
for each element $\alpha\in k^{\times}$ prime to $p$.


(ii)
If $p\not\equiv 2\bmod{3}$ (\resp $p\not\equiv 3\bmod{4}$),
then
$a'_{\sigma}\in\{0,6,12\}$ (\resp $a'_{\sigma}\in\{0,4,8,12\}$)
for any $\sigma\in\Gal(k/\Q)$.

(iii)
$\nu^{12}|_{I_{\mfp^{\sigma^{-1}}}}=\theta_p^{a'_{\sigma}}$
for any $\sigma\in\Gal(k/\Q)$.

\end{lem}

\begin{lem}
\label{type23epsilon'}

Under the situation in Lemma \ref{epsilon'} (2), take a prime number
q different from $p$. Assume that $q$ splits completely in $k$,
and take a prime $\mfq$ of $k$ lying over $q$.
Let $\alpha\in\cO_k\setminus\{0\}$ be an element such that
$\mfq^{h_k}=\alpha\cO_k$.
If $\alpha^{\varepsilon'}=\beta^{12h_k}$
for an  element $\beta\in\cFR(q)$,
then $\varepsilon'$ is of one of the following types.

Type 2:
$\varepsilon'=\displaystyle\sum_{\sigma\in\Gal(k/\Q)}6\sigma$ and $p\equiv 3\bmod{4}$.

Type 3:
$k$ contains $\Q(\beta)$, and
$\varepsilon'=\displaystyle\sum_{\sigma\in\Gal(k/\Q(\beta))}12\sigma$
or $\displaystyle\sum_{\sigma\not\in\Gal(k/\Q(\beta))}12\sigma$.

\end{lem}

Define the finite sets
\noindent
$\cM'_1(k):=$
$$\Set{(\mfq,\varepsilon'_0,\beta_{\mfq})|
\mfq\in \cS,\ \varepsilon'_0=\sum_{\sigma\in\Gal(k/\Q)}a'_{\sigma}\sigma
\text{ with $a'_{\sigma}\in\{0,4,6,8,12 \}$},\ 
\beta_{\mfq}\in\cFR(\N(\mfq))},$$

\noindent
$\cM'_2(k):=\Set{\Norm_{k(\beta_{\mfq})/\Q}(\alpha_{\mfq}^{\varepsilon'_0}-\beta_{\mfq}^{12h_k})\in\Z|
(\mfq,\varepsilon'_0,\beta_{\mfq})\in\cM'_1(k)}\setminus\{0\}$,

\noindent
$\cN'_0(k):=\Set{\text{$l$ : prime number}|\text{$l$ divides some integer $m\in\cM'_2(k)$}}$,


\noindent
$\cN'_1(k):=\cN'_0(k)\cup\cT(k)\cup\Ram(k)$.

\begin{thm}
\label{type23nu}

Assume that $k$ is Galois over $\Q$.
If $p\not\in\cN'_1(k)$
(and if $p$ does not divide $d$), then the character
$\nu:\G_k\longrightarrow \F_p^{\times}$
is of one of the following types.

Type 2:
$\nu^{12}=\theta_p^6$ and $p\equiv 3\bmod{4}$.

Type 3:
There is an imaginary quadratic field $L$ satisfying the following
two conditions.


\noindent
(a)
The Hilbert class field $H_L$ of $L$ is contained in $k$.

\noindent
(b)
There is a prime $\mfp_L$ of $L$ lying over $p$
such that
$\nu^{12}(\mfa)\equiv\delta^{12}\bmod{\mfp_L}$ holds
for any fractional ideal $\mfa$ of $k$ prime to $p$.
Here $\delta$ is any element of $L$ such that
$\Norm_{k/L}(\mfa)=\delta\cO_L$.

\end{thm}

From now to Lemma \ref{q/p-1nu}, assume that $k$ is Galois over $\Q$.

\begin{prop}
\label{type2nuN2}

Suppose 
$p\not\in\cN'_1(k)$.
If $\nu$ is of type 2, then $p\in\cN_2(k)$.

\end{prop}

Proposition \ref{type2nuN2} follows from the following two lemmas.

\begin{lem}
\label{type2nu}

Suppose 
$p\not\in\cN'_1(k)$.
If $\nu$ is of type 2, then
there is a character $\psi':\G_k\longrightarrow\F_p^{\times}$
such that $\psi'^6=1$ and
$\nu=\psi'\theta_p^{\frac{p+1}{4}}$.

\end{lem}

\begin{lem}
\label{q/p-1nu}

Suppose 
$p\not\in\cN'_1(k)$.
Assume that $\nu$ is of type 2.
Let $q<\frac{p}{4}$ be a prime number
which splits completely in $k$.
Then we have
$\left(\frac{q}{p}\right)=-1$.

\end{lem}

\noindent
(Proof of Theorem \ref{mainthm0})

Let $k$ be a quadratic field which is not an imaginary quadratic
field of class number one, and let $p$ be a prime number not dividing $d$.
Take a point $x\in M_0^B(p)(k)$.

(1)
Suppose $B\otimes_{\Q}k\cong\M_2(k)$.

(1-i)
Assume $\Aut(x)\ne\{\pm 1\}$
or $\Aut(x')\not\cong\Z/4\Z$.
Then $x$ is represented by a triple $(A,i,V)$ defined over $k$
by Proposition \ref{fieldM0Bp} (1),
and we have a character
$\nu:\G_k\longrightarrow\F_p^{\times}$ as in (2.2). 
Assume $p\not\in\cN'_1(k)$.
Then the character $\nu$ is of type 2 or type 3
in Theorem \ref{type23nu}.
But type 3 is impossible since
$k$ is a quadratic field which is not an imaginary quadratic field of class number one.
Therefore $\nu$ is of type 2, and so $p\in\cN_2(k)$ 
by Proposition \ref{type2nuN2}.

(1-ii)
Assume otherwise (i.e. $\Aut(x)=\{\pm 1\}$
and $\Aut(x')\cong\Z/4\Z$).
Then $x$ is represented by a triple $(A,i,V)$ defined over a quadratic extension of $k$
by Proposition \ref{fieldM0Bp} (2),
and we have a character
$\varphi:\G_k\longrightarrow\F_p^{\times}$ as in (\ref{phi}).
By Theorem \ref{type23phi} and Proposition \ref{type2},
we have $p\in\cN_1(k)\cup\cN_2(k)\cup\{5,7,13\}$.

(2)
Suppose $B\otimes_{\Q}k\not\cong\M_2(k)$.
Further assume that $x$ is not an elliptic point of order $2$ or $3$;
this implies $\Aut(x)=\{\pm 1\}$.
By the same argument as in (1-ii),
we have $p\in\cN_1(k)\cup\cN_2(k)\cup\{5,7,13\}$.

\qed

\section{Examples of points on Shimura curves of genus zero}
\label{ex}

The genus of the Shimura curve $M^B$ is zero if and only if
$d\in\{6,10,22\}$ (\cite[Lemma 3.1, p.168]{A}).
The defining equations of such $M^B$'s are the following by \cite[Theorem 1-1, p.279]{K}.

\begin{equation*}
\begin{cases}
d=6\ :\ x^2+y^2+3=0,\\
d=10\ :\ x^2+y^2+2=0,\\
d=22\ :\ x^2+y^2+11=0.
\end{cases}
\end{equation*}

\noindent
In these cases, for a number field $k$ the condition $M^B(k)\ne\emptyset$
implies $M^B\otimes_{\Q}k\cong\pP^1_k$, and so
$\sharp M^B(k)=\infty$.

\begin{prop}
\label{MBknotempty}

Let $k$ be a quadratic field and let $d=6$ (\resp $10$, \resp $22$).

(1)
We have $M^B(k)\ne\emptyset$ if and only if $k$ is imaginary
and $3$ (\resp $2$, \resp $11$) does not split in $k$.

(2)
We have $B\otimes_{\Q}k\cong\M_2(k)$ if and only if
neither $2$ nor $3$ (\resp neither $2$ nor $5$, \resp neither $2$ nor $11$)
splits in $k$.

\end{prop}

\begin{proof}

(1)
Obviously $M^B(\R)=\emptyset$ in these cases.
Let $p$ be a prime number.
For $d=6$ (\resp $10$, \resp $22$), we have
$M^B(\Q_p)=\emptyset$ if and only if $p=3$ (\resp $2$, \resp $11$).
For any quadratic extension $L$ of $\Q_p$, we have
$M^B(L)\ne\emptyset$ in these cases.
Then we have done by Hasse principle.

(2)
See the proof of Lemma \ref{fieldofdef}.

\end{proof}

Concerning points over imaginary quadratic fields on $M^B$ of genus zero,
we have the following examples.

(Case $d=6$).

$\bullet$
For $k\in\{\Q(\sqrt{-13}), \Q(\sqrt{-21}), \Q(\sqrt{-37})\}$,
we have
$h_k\ne 1$, $B\otimes_{\Q}k\cong\M_2(k)$ and $M^B(k)\ne\emptyset$.

$\bullet$
For $k\in\{\Q(\sqrt{-31}), \Q(\sqrt{-39})\}$,
we have
$h_k\ne 1$, $B\otimes_{\Q}k\not\cong\M_2(k)$ and $M^B(k)\ne\emptyset$.

(Case $d=10$).

$\bullet$
For $k\in\{\Q(\sqrt{-17}), \Q(\sqrt{-22}), \Q(\sqrt{-33}), \Q(\sqrt{-38})\}$,
we have
$h_k\ne 1$, $B\otimes_{\Q}k\cong\M_2(k)$ and $M^B(k)\ne\emptyset$.

$\bullet$
For $k\in\{\Q(\sqrt{-6}), \Q(\sqrt{-34}), \Q(\sqrt{-51})\}$,
we have
$h_k\ne 1$, $B\otimes_{\Q}k\not\cong\M_2(k)$ and $M^B(k)\ne\emptyset$.

(Case $d=22$).

$\bullet$
For $k\in\{\Q(\sqrt{-5}), \Q(\sqrt{-33}), \Q(\sqrt{-37})\}$,
we have
$h_k\ne 1$, $B\otimes_{\Q}k\cong\M_2(k)$ and $M^B(k)\ne\emptyset$.

$\bullet$
For $k\in\{\Q(\sqrt{-15}), \Q(\sqrt{-23}), \Q(\sqrt{-31}), \Q(\sqrt{-47})\}$,
we have
$h_k\ne 1$, $B\otimes_{\Q}k\not\cong\M_2(k)$ and $M^B(k)\ne\emptyset$.

%

\section{Consequence to Galois images}
\label{galim}

\begin{thm}
\label{irred}
Let $k$ be an imaginary quadratic field with $h_k\geq 2$.
Then for any QM-abelian surface $(A,i)$ by $\cO$ over $k$ and for any prime number $p$
not dividing $d$
with $p\not\in\cN'_1(k)\cup\cN_2(k)$, 
the representation
$$\rhob_{A,p}:\G_k\longrightarrow\GL_2(\F_p)$$
is irreducible.

\end{thm}

\begin{proof}

Suppose $p\nmid d$ and that $\rhob_{A,p}$ is reducible.
Then there is a $1$-dimensional sub-representation of $\rhob_{A,p}$;
let $\nu:\G_k\longrightarrow\F_p^{\times}$ be its associated character.
Then by repeating exactly the same argument as in (1-i) in the proof of
Theorem \ref{mainthm0} at the end of \S \ref{charII},
we conclude
$p\in\cN'_1(k)\cup\cN_2(k)$. 

\end{proof}

\section{Application to a finiteness conjecture on abelian varieties}
\label{apply}

For a number field $k$ and a prime number $p$, let $\ktil_p$ denote the maximal pro-$p$
extension of $k(\pmu_p)$ in $\Qb$ which is unramified away from $p$,
where $\pmu_p$ is the group of $p$-th roots of unity in $\Qb$.
For a number field $k$, an integer $g\geq 0$ and a prime number $p$,
let
$\sA(k,g,p)$ denote the set of $k$-isomorphism classes $[A]$ of abelian varieties
$A$ over $k$, of dimension $g$, which satisfy
$$k(A[p^{\infty}])\subseteq\ktil_p,$$
where $k(A[p^{\infty}])$ is the field generated over $k$ by the $p$-power torsion of $A$.
We know that $\sA(k,g,p)$ is a finite set
(\cite[1. Theorem, p.309]{Z}, \cf \cite[Satz 6, p.363]{Fa}).
For fixed $k$ and $g$, define the set
$$\sA(k,g):=\{([A],p)\mid [A]\in\sA(k,g,p)\}.$$
We have the following finiteness conjecture on abelian varieties.

\begin{conj}[{\cite[Conjecture 1, p.1224]{RT}}]
\label{RTconj}

Let $k$ be a number field, and let $g\geq 0$ be an integer.
Then the set $\sA(k,g)$
is finite.

\end{conj}

As an application of Theorem \ref{X0p}, the following is known.

\begin{thm}[{\cite[Theorem 2, p.1224 and Theorem 4, p.1227]{RT}}]
\label{A(Q,1)}

Let $k$ be $\Q$ or a quadratic field which is not an imaginary quadratic field
of class number one.
Then the set $\sA(k,1)$ is finite.

\end{thm}

Here we study the case where $g=2$.
For an indefinite quaternion division algebra $B$ over $\Q$,
let $\sA(k,2,p)_B$ be the subset of $\sA(k,2,p)$
consisting of abelian varieties $A$ over $k$ whose endomorphism algebra $\End_k(A)$ contains
a maximal order $\cO$ of $B$ as a subring.
Define also the set
$$\sA(k,2)_B:=\{([A],p)\mid [A]\in\sA(k,2,p)_B\},$$
which is a subset of $\sA(k,2)$.
If one of the following two conditions
is satisfied, we know that the set $\sA(k,2)_B$ is empty
(\cite[Theorem 0, p.136]{Sh} and \cite[Theorem (1.1), p.93]{J}).

\noindent
(i) $k$ has a real place.

\noindent
(ii) $B\otimes_{\Q}k\not\cong\M_2(k)$.

As an application of Theorem \ref{irred}, we have the following.

\begin{thm}
\label{RTB}

Let $k$ be an imaginary quadratic field with $h_k\geq 2$.
Then the set $\sA(k,2)_B$ is finite.

\end{thm}

\begin{proof}

Let $([A],p)$ be an element of  $\sA(k,2)_B$ with $p\nmid\disc B$.
Then the Galois representation $\rhob_{A,p}$ is conjugate to the form
$\left(\begin{matrix} \theta_p^a & * \\ 0 & \theta_p^b\end{matrix}\right)$
where $\theta_p:\G_k\longrightarrow\F_p^{\times}$ is the mod $p$ cyclotomic character
(\cite[Lemma 3, p.1225]{RT};
notice that this lemma holds if $\chib$ is injective).
In particular, the representation $\rhob_{A,p}$ is reducible.
By Theorem \ref{irred},
we know that there are only finitely many such prime numbers $p$.
Since each set $\sA(k,2,p)_B$ is finite, the set $\sA(k,2)_B$ is also finite.

\end{proof}

\begin{rmk}
\label{QMdetcyclo}

In the proof of Theorem \ref{RTB}, we have
$\det\rhob_{A,p}=\theta_p$
(\cite[Proposition 1.1 (2), p.300]{Oh}).
So, in fact
$\rhob_{A,p}$ is conjugate to the form
$\left(\begin{matrix} \theta_p^a & * \\ 0 & \theta_p^{1-a}\end{matrix}\right)$
just as in the case of an elliptic curve
(\cf \cite[(6), p.1226]{RT}).

\end{rmk}

Let $\cQM$ be the set of isomorphism classes of indefinite quaternion division
algebras over $\Q$.
Define the set
$$\sA(k,2)_{\cQM}:=\bigcup_{B\in\cQM}\sA(k,2)_B,$$
which is a subset of $\sA(k,2)$.
As a corollary of Theorem \ref{RTB}, we know the following.

\begin{cor}
\label{RTQM}

Let $k$ be an imaginary quadratic field with $h_k\geq 2$.
Then the set $\sA(k,2)_{\cQM}$ is finite.

\end{cor}

\begin{proof}

Let $([A],p)$ be an element of $\sA(k,2)_{\cQM}$.
Then there is an indefinite quaternion division algebra $B$ over $\Q$
such that $\End_k(A)$ contains a maximal order $\cO$ of $B$ as a subring.
If we let $i:\cO\hookrightarrow\End_k(A)$
be the inclusion, the pair $(A,i)$ is a QM-abelian surface by $\cO$ over $k$
and it determines a point of $M^B(k)$.
In this case we have
$B\otimes_{\Q}k\cong\M_2(k)$
by \cite[Theorem (1.1), p.93]{J}.
Further we find that there are only finitely many such $B$'s
by \cite[Theorem 6.6, p.111]{J}.

\end{proof}

\begin{rmk}
\label{Ozeki}

Conjecture \ref{RTconj} is solved for any $K$ and $g$
when restricted to semi-stable abelian varieties (\cite[Corollary 4.5, p.2392]{O1})
or abelian varieties
with abelian Galois representations (\cite[Theorem 1.2]{O2}).

\end{rmk}




\def\bibname{References}

(Keisuke Arai)
Department of Mathematics, School of Engineering,
Tokyo Denki University,
5 Senju Asahi-cho, Adachi-ku, Tokyo 120-8551, Japan

\textit{E-mail address}: \texttt{araik@mail.dendai.ac.jp}

\vspace{3mm}

(Fumiyuki Momose)
Department of Mathematics, Faculty of Science and Engineering,
Chuo University,
1-13-27 Kasuga, Bunkyo-ku, Tokyo 112-8551, Japan



\begin{thebibliography}{SGA5}

\bibitem{A} \textit{K.\ Arai},
On the Galois images associated to QM-abelian surfaces,
Proceedings of the Symposium on Algebraic
Number Theory and Related Topics, 165--187, 
RIMS K\^{o}ky\^{u}roku Bessatsu, {\bf B4}, Res. Inst. Math. Sci. (RIMS), Kyoto, 2007.

\bibitem{A2} \textit{K.\ Arai},
Galois images and modular curves,
Proceedings of the Symposium on Algebraic
Number Theory and Related Topics, 145--161, 
RIMS K\^{o}ky\^{u}roku Bessatsu, {\bf B32}, Res. Inst. Math. Sci. (RIMS), Kyoto, 2012.

\bibitem{Ba} \textit{A.\ Baker},
Linear forms in the logarithms of algebraic numbers, 
Mathematika {\bf 13} (1966), 204--216.

\bibitem{Be} \textit{A.\ Besser},
CM cycles over Shimura curves,
J. Algebraic Geom. {\bf 4} (1995), no. 4, 659--691.

\bibitem{Bu} \textit{K.\ Buzzard}, 
Integral models of certain Shimura curves, 
Duke Math. J. {\bf 87} (1997), no. 3, 591--612.

\bibitem{DD} \textit{P.\ D\`{e}bes, J.-C.\ Douai},
Algebraic covers: field of moduli versus field of definition,
Ann. Sci. Ecole Norm. Sup. (4) {\bf 30} (1997), no. 3, 303--338.

\bibitem{DR} \textit{P.\ Deligne, M.\ Rapoport},
Les sch\'{e}mas de modules de courbes elliptiques,
Modular functions of one variable, II, 
143--316. Lecture Notes in Math., Vol. {\bf 349}, Springer, Berlin, 1973.

\bibitem{Fa} \textit{G.\ Faltings},
Endlichkeitss\"{a}tze f\"{u}r abelsche Variet\"{a}ten \"{u}ber Zahlk\"{o}rpern,
Invent. Math. {\bf 73} (1983), no. 3, 349--366.

\bibitem{J} \textit{B.\ Jordan},
Points on Shimura curves rational over number fields,
J. Reine Angew. Math. {\bf 371} (1986), 92--114. 

\bibitem{K} \textit{A.\ Kurihara},
On some examples of equations defining Shimura curves
and the Mumford uniformization,
J. Fac. Sci. Univ. Tokyo Sect. IA Math. {\bf 25} (1979), no. 3, 277--300. 



\bibitem{Ma} \textit{B.\ Mazur},
Rational isogenies of prime degree (with an appendix by D. Goldfeld), 
Invent. Math. {\bf 44} (1978), no. 2, 129--162.

\bibitem{Mo} \textit{F.\ Momose},
Isogenies of prime degree over number fields,
Compositio Math. {\bf 97} (1995), no. 3, 329--348.

\bibitem{Mu} \textit{D.\ Mumford},
Abelian varieties,
Tata Institute of Fundamental Research Studies in Mathematics,
No. 5 Published for the Tata Institute of Fundamental Research, Bombay; Oxford University Press, London 1970.




\bibitem{Oh} \textit{M.\ Ohta}, 
On $l$-adic representations of Galois groups obtained from certain 
two-dimensional abelian varieties, 
J. Fac. Sci. Univ. Tokyo Sect. IA Math. {\bf 21} (1974), 299--308.

\bibitem{O1} \textit{Y.\ Ozeki},
Non-existence of certain Galois representations with a uniform tame inertia weight,
Int. Math. Res. Not. {\bf 2011} (2011), no. 11, 2377--2395. 

\bibitem{O2} \textit{Y.\ Ozeki},
Non-existence of certain CM abelian varieties with prime power torsion,
preprint, available at the web page
(http://arxiv.org/pdf/1112.3097.pdf).

\bibitem{RT} \textit{C.\ Rasmussen, A.\ Tamagawa},
A finiteness conjecture on abelian varieties with constrained prime power torsion,
Math. Res. Lett. {\bf 15} (2008), no. 6, 1223--1231. 

\bibitem{Ray} \textit{M.\ Raynaud},
Sch\'{e}mas en groupes de type $(p,\dots, p)$,
Bull. Soc. Math. France {\bf 102} (1974), 241--280.

\bibitem{Se} \textit{J.-P.\ Serre},
Propri\'{e}t\'{e}s galoisiennes des points d'ordre fini des 
courbes elliptiques, 
Invent. Math. {\bf 15} (1972), no. 4, 259--331. 


\bibitem{Shimizu} \textit{H.\ Shimizu}, 
Hokei kans\={u}. I--III (Japanese) [Automorphic functions. I--III] 
Second edition.
Iwanami Shoten Kiso S\={u}gaku [Iwanami Lectures on Fundamental Mathematics], 
8. Dais\={u} [Algebra], vii. Iwanami Shoten,
Tokyo, 1984. 

\bibitem{Sh0} \textit{G.\ Shimura},
Introduction to the arithmetic theory of automorphic functions,
Kano Memorial Lectures, No. 1. Publications of the Mathematical Society of Japan, No. 11.
Iwanami Shoten, Publishers, Tokyo; Princeton University Press, Princeton, N.J., 1971.

\bibitem{Sh} \textit{G.\ Shimura},
On the real points of an arithmetic quotient of a bounded symmetric domain,
Math. Ann. {\bf 215} (1975), 135--164.

\bibitem{Si} \textit{J.\ Silverman},
The arithmetic of elliptic curves.
Second edition. Graduate Texts in Mathematics, 106. Springer, Dordrecht, 2009.

\bibitem{St} \textit{H.\ M.\ Stark},
A complete determination of the complex quadratic fields of class-number one,
Michigan Math. J. {\bf 14} (1967) 1--27.

\bibitem{We} \textit{A.\ Weil},
Basic number theory,
Reprint of the second (1973) edition. Classics in Mathematics. Springer-Verlag, Berlin, 1995.

\bibitem{Z} \textit{Y.\ G.\ Zarhin},
A finiteness theorem for unpolarized abelian varieties over number fields
with prescribed places of bad reduction,
Invent. Math. {\bf 79} (1985), no. 2, 309--321.

\end{thebibliography}
\end{document}